\documentclass[a4paper,11pt]{amsart}

\usepackage{xcolor}
\usepackage[utf8]{inputenc}
\usepackage{amsmath, amsthm, amssymb, mathtools}
\usepackage{amsfonts}
\usepackage{bm}
\usepackage{cancel}
\usepackage{color}
\usepackage{subfig}
\usepackage{stmaryrd}
\usepackage{xparse} 
\usepackage{nomencl} 
\usepackage[colorlinks]{hyperref}
\usepackage{cleveref}
\usepackage{empheq} 
\usepackage{enumerate}

\usepackage{comment}
\usepackage[backend=biber,style=numeric, firstinits=true, maxalphanames=4, maxnames=99,minnames=99]{biblatex}
\renewbibmacro{in:}{}
\DeclareSourcemap{
  \maps[datatype=bibtex]{
    \map{
      \step[fieldsource=journal, match={Numerische Mathematik}, replace={Numer. Math.} ]
    }
     \map{
      \step[fieldsource=journal, 
     match={SIAM Journal on Numerical Analysis}, replace={SIAM J. Numer. Anal.} ]
    }
    \map{
      \step[fieldsource=journal, 
     match={Journal of Computational Physics}, replace={J. Comput. Phys.} ]
    } 
    \map{
      \step[fieldsource=journal, 
     match={Acta numerica}, replace={Acta Numer.} ]
    }
  }
}

\usepackage{textcomp} 
\usepackage{graphicx}
\usepackage{float}
\usepackage[export]{adjustbox}

\hypersetup{linkcolor=blue}
\makenomenclature 

\numberwithin{equation}{section}

\DeclarePairedDelimiter{\norm}{\lVert}{\rVert}
\DeclarePairedDelimiter{\jumpop}{\llbracket}{\rrbracket}
\DeclareSymbolFont{matha}{OML}{txmi}{m}{it}
\DeclareMathSymbol{\varv}{\mathbf}{matha}{118}

\theoremstyle{definition}

\newtheorem{Theorem}{Theorem}[section]

\newtheorem{Lemma}{Lemma}[section]

\newtheorem{Remark}{Remark}[section]

\newcommand{\restc}{\mathbin{\vrule height 1.6ex depth 0pt width
0.13ex\vrule height 0.13ex depth 0pt width 1.3ex}}

\newcommand\restr[2]{{
  \left.\kern-\nulldelimiterspace 
  #1 
  \vphantom{\big|} 
  \right|_{#2} 
  }}

\NewDocumentCommand{\dgal}{sO{}m}{%
  \IfBooleanTF{#1}
    {\dgalext{#3}}
    {\dgalx[#2]{#3}}%
}

\NewDocumentCommand{\dgalext}{m}{%
  \sbox0{%
    \mathsurround=0pt 
    $\left\{\vphantom{#1}\right.\kern-\nulldelimiterspace$%
  }%
  \sbox2{\{}%
  \ifdim\ht0=\ht2
    \{\kern-.625\wd2 \{#1\}\kern-.625\wd2 \}%
  \else
    \left\{\kern-.7\wd0\left\{#1\right\}\kern-.7\wd0\right\}%
  \fi
}

\NewDocumentCommand{\dgalx}{om}{%
  \sbox0{\mathsurround=0pt$#1\{$}%
  \sbox2{\{}%
  \ifdim\ht0=\ht2
    \{\kern-.625\wd2 \{#2\}\kern-.625\wd2 \}%
  \else
    \mathopen{#1\{\kern-.7\wd0 #1\{}
    #2
    \mathclose{#1\}\kern-.7\wd0 #1\}}
  \fi
}

\newcommand{\R}{{\mathbb{R}}}

\crefformat{section}{\S#2#1#3} 
\crefformat{subsection}{\S#2#1#3}
\crefformat{subsubsection}{\S#2#1#3}

\numberwithin{equation}{section}

\def\<{{\langle }}
\def\>{{\rangle }}

\newcommand{\J}[1]{{  [\! [#1]\! ]  }}   

\newcommand{\Av}[1]{{  \{\!\!\{#1\}\!\!\}  }}

\date{\today}

\title[DG for non-convex energies]{\bf Convergence of Discontinuous Galerkin Methods for Quasiconvex and  Relaxed Variational Problems
}

\author[G.\ Grekas]{G.\ Grekas $^{1}$}
\thanks{$^1$Computer, Electrical, Mathematical  Sciences and Engineering Division, KAUST}
\thanks {$^{2}$Department of Mathematics, University of Sussex.}
\thanks {$^{3}$Institute of Applied \& Computational Mathematics, Foundation for Research \& Technology-Hellas}
\thanks {$^{4}$Department of Mathematics and Applied Mathematics, University of Crete.}
\thanks {$^{5}$Faculty of Civil Engineering-Czech Technical University of Prague}
\author[K.\ Koumatos]{K.\ Koumatos $^{2}$}
\author[C.\ Makridakis]{C.\ Makridakis $^{2,3,4}$}
\author[A.\ Vikelis]{A.\ Vikelis $^{5}$}

\begin{document}

\maketitle

\begin{abstract}
In this work, we establish that discontinuous Galerkin methods are capable of producing reliable approximations for a broad class of nonlinear variational problems. In particular, we demonstrate that these schemes provide essential flexibility by removing inter-element continuity while also guaranteeing convergent approximations in the quasiconvex case. Notably, quasiconvexity is the weakest form of convexity pertinent to elasticity. Furthermore, 
we show that in the non-convex case discrete minimisers converge to minimisers of the relaxed problem. In this case, the minimisation problem corresponds to the energy defined by the \emph{quasiconvex envelope} of the original energy. Our approach covers all discontinuous Galerkin formulations known to converge for convex energies. This work addresses an open challenge in the vectorial calculus of variations: developing and rigorously justifying numerical schemes capable of reliably approximating nonlinear energy minimization problems with potentially singular solutions, which are frequently encountered in materials science.
\end{abstract}


\section{Introduction}

In this work, we establish that discontinuous Galerkin methods are capable of producing reliable approximations for a broad class of nonlinear variational problems. Our aim is twofold: first, to show that these methods  are applicable to variational problems involving \emph{quasiconvex energies}; and second, to establish that, even in the absence of convexity assumptions, the approximations converge to the minimisers of the \emph{relaxed} problem. In this case, the minimisation problem corresponds to the energy defined by the \emph{quasiconvex envelope} of the original energy.


By addressing this, we contribute to a long-standing algorithmic and theoretical challenge in the vectorial calculus of variations: the development and rigorous justification of numerical schemes capable of reliably approximating nonlinear energy minimisation problems with possibly singular solutions, which commonly arise in materials science.


The insightful observation by J.M. Ball in \cite{Ball_computation_2001}—``minimisers of variational problems may have singularities, \emph{but natural numerical schemes may fail to detect them}''—remains highly relevant today.  As discussed in detail in \textcite{Ball_computation_2001} and \textcite{Cc_review_2001}, although standard finite element methods provide a natural discrete framework for seeking minimisers of such problems, the inter-element continuity imposed on the discrete spaces appears to lack sufficient flexibility to make these methods the optimal choice. This limitation is particularly evident in the computation of minimisers for variational problems where the \emph{Lavrentiev phenomenon} may occur. As first noted in \cite{Ball_Knowles_1987}, standard finite element methods can converge to energy values that are strictly greater than the exact energy of the problem, leading to physically irrelevant approximations. 


In the finite element literature, there are alternatives to standard finite elements that offer greater flexibility, such as nonconforming finite elements and discontinuous Galerkin methods. The spaces introduced by \textcite{Cr_Rav_original_1973}, originally developed for the Stokes and Navier-Stokes equations, were employed by \textcite{Ortner_CR_2011} (see also \cite{Ortner_Praet_CR_2011} and \cite{Ortner_CR2_2022}) to demonstrate that these spaces provide sufficient flexibility to overcome the limitations of standard methods, particularly in addressing the Lavrentiev gap phenomenon.

Discontinuous Galerkin methods, which are completely discontinuous, have become especially popular for nonlinear PDEs. While these methods offer high-order approximation capabilities, 
and great flexibility, they require specially constructed energy functionals, which are often nontrivial to design  and analyse \cite{ten2006discontinuous}, \cite{buffa2009compact}, \cite{GMKV2023}. 
 $\Gamma$-convergence for convex energies associated with the energy functional proposed in \cite{ten2006discontinuous} was established in \textcite{buffa2009compact}, while the authors of this paper derived $\Gamma$-convergence for the functional introduced in  \cite{GMKV2023}.

%
Despite the progress outlined above, allowing possibly noncontinuous discrete approximations comes at a cost: the analytical treatment of energies involving the discrete gradient heavily relies on the convexity of 
 $W.$ Specifically, due to the absence of a true gradient structure in the discrete gradient —i.e., 
${\rm curl}\, G_h(u_h)\neq 0,$ $G_h(u_h)$ being an appropriate ``discrete gradient"---  the lower semicontinuity of the energy term is guaranteed for convex integrands 
$W,$ but remains uncertain when 
$W$ is  quasiconvex. This represents a significant technical challenge that hinders further advancements in the design and analysis of approximation methods applicable to elasticity and other variational problems.

In this work, we resolve this issue by showing that discontinuous Galerkin schemes not only provide the necessary flexibility by eliminating inter-element continuity but also ensure convergent approximations in the quasiconvex case.
To compensate the lack of differential structure and regain lower semicontinuity, we transition to the limit of reconstructed approximations from \textcite{karakashian2003posteriori} and demonstrate that the associated error terms can be adequately controlled. Furthermore, 
we show that in the non-convex case discrete minimisers converge to minimisers of the relaxed problem.

\subsubsection*{Plan of the paper.} For $\Omega\subset\mathbb{R}^d$, we consider the minimisation problem associated with the energy
\begin{equation}\label{eq:energyintro}
\mathcal{E}(u) = \int_\Omega W(\nabla u(x))\,dx,\quad\text{where}\,\,\, u:\Omega\to \R^N,\,\,u = u_0,\mbox{ on }\partial\Omega, 
\end{equation}
over an admissible set of functions which is a subset of a Sobolev space $W^{1,p}$ and  where  $W:\R^{N\times d}\to\R.$ Detailed assumptions on the corresponding energy minimisation problem are in Section 3.
We start with some preliminary material in Section \ref{Se:Prelim} including standard notation on the discontinuous Galerkin spaces $V_h(\Omega)$. Lemma \ref {lemma:tekar} is an extension in the $L^p$ setting of the classical result in \cite{karakashian2003posteriori} regarding the error between 
$u_h \in V_h(\Omega)$ and its   continuous reconstruction  $w_h \in V_h(\Omega) \cap W^{1,p}(\Omega). $ This is a crucial tool in our analysis.  

In Section \ref{Se:Convergence}
 we introduce our discontinuous Galerkin method \cite{GMKV2023} and we show that, in the absence of convexity, our discrete energies $\Gamma$-converge to the relaxed problem. Similarly with \cite{GMKV2023}, a penalty term ensures the discrete stability, but in this case it is appropriately  modified to include also energies which are not convex.  Our results cover both discontinuous Galerkin formulations known to converge for convex $W.$ The proof for the method based on discrete gradients \cite{ten2006discontinuous} is postponed  to the Section \ref{Se:discrG}. The discrete gradients are defined through standard lifting operators, defined in the same section. Although   the use of the lifting operators and the discrete gradients was the main tool in the analysis in the convex case, as mentioned, this is no longer true 
 if we would like to consider   quasiconvex integrands. The main obstacle in proving the convergence of DG schemes for quasiconvex energies lies in the fact that the finite element function $\nabla_h y_h$ used to approximate the gradient of an admissible function $\nabla y$ is not itself a gradient, i.e. ${\rm curl} \nabla_h y_h \neq 0$, and hence in general the lower semicontinuity of $\mathcal{E}$ with respect to weak convergence in $W^{1,p}$ is not true. That is, for a quasiconvex $W$, the inequality
\[
\liminf_{h\to 0}\int_\Omega W(\nabla_h y_h) \geq \int_\Omega W(\nabla y)
\] 
does not hold in general when
\[
\nabla_h y_h \rightharpoonup \nabla y,\quad \mbox{ in }L^p\mbox{ as }h\to 0.
\]
The main results of this paper are Theorems \ref{thm:main} and \ref{thm:main_dGr} where we show that the discrete minimisers 
of disontinuous Galerkin schemes of \cite{GMKV2023} and  \cite{ten2006discontinuous}  converge to a minimiser of 
the relaxed problem defined as
\[
\mathcal{E}^{rel}(y) = \int_\Omega W^{qc}(\nabla y(x))\,dx
\]
where $W^{qc}$ denotes the quasiconvex envelope of $W$, i.e. the largest quasicnovex function that lies below $W$. In particular, if $W$ is quasiconvex, then $W = W^{qc},$ and as a consequnce we establish convergence of the discrete mimimisers to a minimiser of the continuous quasiconvex energy.  

In Section \ref{sec:computations}, we provide detailed numerical evidence using specifically designed examples. We begin with a problem involving polyconvex energy and thoroughly analyze the impact of the new penalty term introduced in this work (see Section 3.1) in comparison to the penalty term employed in \cite{GMKV2023}.


Subsequently, we examine a frame-indifferent two-well strain energy function where, under specific boundary conditions, a phase mixture emerges. This energy leads to increasingly fine microstructures as the mesh size 
$h$  is reduced, approaching the infimum of the total potential energy. An intriguing aspect of our experiments, consistent with the theoretical results, is the observation that for this problem, minimizing sequences with distinct laminate structures exist. Nevertheless, all such sequences strongly converge in 
$L^2$   to the same function. This behavior can be attributed to the diminishing amplitude of gradient oscillations, which become negligible as 
 $h$ decreases.


An additional interesting observation is that, by leveraging Dacorogna's formula for the quasiconvex envelope of a general energy density 
$W$, our scheme can approximate 
$W^{qc}(F)$ pointwise for each 
$F.$ This suggests that quasiconvex envelope approximations can be constructed using discrete energy minimisation techniques. While the full implementation of this idea lies beyond the scope of the current paper, we present preliminary results for this pointwise approximation in Figures~\ref{fig:minimizing_seq_def}, 6, 7. As mentioned, these numerical experiments are for non-quasiconvex, double-well potentials 
$W,$ which are typical in modeling phase transitions in solids.

%

\subsubsection*{Further references.} The systematic study of numerical methods for energy minimisation problems arising in elasticity started in the late 80-90's with the works of \textcite{Ball_Knowles_1987},
\textcite{Luskin_Col_Kind_1991},
\textcite{Luskin_Col_nonconvex_1991},
\textcite{Pedregal_num_1996},
\textcite{CcPl_adaptive_scalar_nonc_1998},
\cite{CcPl_phas_trans_elast_2000},
\cite{CcPl_scalar_micros_1997},
\textcite{CcRubiceck_numer_young_micros_2000}. For reviews, open problems and remarks of the effectivity of various numerical approaches, see e.g., \cite{Luskin_acta_microstr_1996}, \cite{Ball_computation_2001}, \cite{Cc_review_2001}, \cite{Pedregal_book_2000}, \cite{Cc_remarks_nonc_2005}, \cite{Bartels_book} and their references.   Early approaches emphasising the benefits of potential flexibility offered by nonconforming or discontinuous methods were discussed in   \cite{Luskin_acta_microstr_1996}, \cite{Gobbert_Prohl_DG_1999}.  Crouzeix-Raviart nonconforming spaces are used in total variation and other convex minimisation problems in \cite{Bartels_TV_2012}, \cite{Bartels_convex_2023}. 
Discontinuous Galerkin schemes for problems with higher gradients were considered in \cite{BonitoNoch_Ntog_2021},  \cite{grekas2022approximations}.
To obtain feasible discrete solutions for the relaxed problem using its definition, it is necessary to compute approximations of polyconvex or quasiconvex envelopes. Such approaches have been explored in \cite{Kru_1998}, \cite{Dolzmann_Walkington_2000}, \cite{BartCc_et_relaxed_micros_2004}, \cite{Bartels_approx_poly_2005}, \cite{Peterseim_et_conv_micr_2024}, \cite{Peterseim_et_2024}, 
 among others, along with the references therein.

\section{Preliminaries.}\label{Se:Prelim}

 \subsubsection{Finite element spaces and trace operators} 
 For simplicity, let us assume that $\Omega$ is a polytope domain, and consider a triangulation $T_h$ of $\Omega$ of mesh size $h>0$. We seek discontinuous Galerkin approximations to \eqref{eq:energyintro}. To this end, we introduce the space of $L^p$ piecewise polynomial functions
\begin{equation}\label{eq:DGspace}
V^q_h(\Omega):=\left\{v\in L^p(\Omega)\,:\,\restr{v}{K}  \in \mathbb{P}_q(K), K \in T_h \right\},
\end{equation}
where $\mathbb{P}_q(K)$ denotes the space of polynomial functions on $K$ of degree $q\in \mathbb{N}$ which is henceforth fixed. Moreover, we adopt the following standard notation:
The \textit{trace} of 
functions in $V^q_h(\Omega)$ belongs to the space
\begin{align}
 T(E_h) := \Pi_{e\in E_h} L^p(e),
\end{align}
where $E_h$ is the set of mesh edges.
The \textit{average} and \textit{jump} operators over
$T(E_h)$ are defined by:
\begin{equation}
 \begin{aligned}
\dgal{ \cdot } :& T(E_h) \mapsto L^p(E_h) \\
  &\dgal{ w } := \left\{
        \begin{array}{ll}
             \frac{1}{2} (\restr{w}{K_{e^+} }  + \restr{w}{K_{e^-} }),
             & \quad \text{for } e \in E_h^i \\
             w, & \quad \text{for } e \in E_h^b,
        \end{array}
    \right.
 \end{aligned}
\label{average_operator}
\end{equation}
\begin{equation}
 \begin{aligned}
 \jumpop{ \cdot }
 :& T(E_h) \mapsto L^p(E_h) \\
  &\jumpop{ v }	 := \restr{v}{K_{e^+} }  - \restr{v}{K_{e^-} },
             & \quad \text{for } e \in E_h^i .
 \end{aligned}
\label{jump_operator}
\end{equation}
We often apply the jump operator on dyadic products $v\otimes n$, for $v\in \R^N$, $n\in \R^d$ and we write
\begin{equation}
 \begin{aligned}
  &\jumpop{ v\otimes n_e }	 := \restr{v\otimes n_{e^+}}{K_{e^+} }  + \restr{v\otimes n_{e^-}}{K_{e^-} },
             & \quad \text{for } e \in E_h^i  
 \end{aligned}
\label{jump_operator_tensor}
\end{equation}
where $(v\otimes n)_{ij} = v_in_j$, $K_{e^+}$, $K_{e^-}$ are the elements that share the internal edge $e$ with corresponding outward pointing normals
$n_{e^+}, n_{e^-}$, and $E_h^i$, $E_h^b$ denote respectively the set of internal and boundary edges.

The space $V^q_h(\Omega)$ is equipped with the norm
\[
\|v\|^p_{W^{1,p}(\Omega,T_h)} = \|v\|^p_{L^p(\Omega)} + |v|^p_{W^{1,p}(\Omega,T_h)},
\]
where the seminorm $|\cdot|_{W^{1,p}(\Omega,T_h)}$ is defined by
\begin{equation}\label{eq:seminorn}
|v|^p_{W^{1,p}(\Omega,T_h)} := \sum_{K\in T_h} \int_K |\nabla v|^p +\sum_{e\in E_h^i} \frac{1}{h^{p-1}}\int_{e} |\jumpop{v}|^p.
\end{equation}

We note that in the above expression $\nabla v$ denotes the gradient, in the classical sense, of the polynomial function $v$ over the set $K$. As a function in $SBV(\Omega)$, $v\in V^q_h(\Omega)$, admits a distributional gradient which is a Radon measure and decomposes as
\begin{align}
Dv = \nabla v\mathcal{L}^d\restc\Omega + \jumpop{v}\otimes n_{e} \mathcal{H}^{d-1}\restc E_h
\label{distrib_gradient}
\end{align}
where $\nabla v$ denotes the approximate gradient. To avoid confusion, we denote the above approximate gradient by $\nabla_h v$, i.e. we write $\nabla_h: V^k_h \to V^{k-1}_h$ for the operator defined by
\[
\restr{\left(\nabla_h v\right)}{K} = \nabla (\restr{v}{K}).
\]
In particular, it holds that
\[
\int_{\Omega} W(\nabla_h u_h) = \sum_{K\in T_h} W(\nabla u_h).
\]

\subsubsection{Approximation results and Poincar\'{e} inequalities in  DG spaces.}

Here we approximate discontinuous functions $u_h \in V_h(\Omega)$ by  continuous piecewise
functions $w_h \in V_h(\Omega) \cap W^{1,p}(\Omega)$ and prove their error estimates by 
generalising the case of $p=2$ presented in \cite[Theorem 2.2]{karakashian2003posteriori}. 
From this result we derive a  Poincar\'{e} inequality for functions in $V_h(\Omega)$ 
while these estimates are constituents for the $\Gamma-$convergence result when the energy function $W$ is quasiconvex. First we generalise the algebraic inequality of \cite[Lemma 2.2]{karakashian2003posteriori}
\begin{Lemma}\label{lemma:tekar} Given $n$ real numbers $c_1, c_2, ..., c_n$, let $m = \frac{1}{n} \sum_{i=1}^{n} c_i$.
If $r \ge1$, then 
\begin{align*}
\sum_{i=1}^n |c_i - m|^r \le C \sum_{i=1}^{n-1} |c_{i+1} - c_i|^r,
\end{align*}
where $C$ depends only on $n$ and $r$. 
\end{Lemma}
\begin{proof}
Here we follow partially the proof of \cite[Lemma 2.2]{karakashian2003posteriori}.  For $r\ge1$
one can use  Jensen's inequality for an upper bound
\begin{align*}
|c_j - m|^r = \frac{1}{N^r} \left| \sum_{i=1}^n (c_j -c_i) \right|^r \lesssim \sum_{i=1}^n |c_j - c_i|^r.
\end{align*}
Now, summing over $j$  we deduce that
\begin{equation*}
\begin{aligned}
&\sum_{j=1}^n |c_j - m|^r \lesssim \sum_{j=1}^n \sum_{i=j+1}^n |c_j -c_i|^r = 
\sum_{j=1}^n \sum_{i=j+1}^n \left| \sum_{k=i}^{j-1}(c_{k+1} -c_k)\right|^r \\
\lesssim & \sum_{j=1}^n \sum_{i=j+1}^n \sum_{k=i}^{j-1} \left|  c_{k+1} -c_k \right|^r \lesssim \sum_{j=1}^n |c_{j+1} - c_j|^r. 
\end{aligned}
\end{equation*}

\end{proof}
Using the above Lemma, we extend the approximation result \cite[Theorem 2.2]{karakashian2003posteriori} for the case $p\geq 2$.
We omit the details of the proof emphasising more on the points where our case differs from the classical case $p=2$.

\begin{Lemma}\label{lemma:karap} Let $u_h \in V_h(\Omega)$, then for any multi-index $\alpha$ with $|\alpha|=$$0,\,1$ there exists  
$w_h\in V_h(\Omega) \cap W^{1,p}(\Omega)$ (continuous finite element space), such that
\begin{align}
\sum_{K \in T_h}\norm{\nabla^\alpha u_h - \nabla^\alpha w_h}_{L^p(K)}^p
\lesssim \sum_{e\in E^i_h} h_e^{1-p|\alpha|} \int_e |\jumpop{u_h}|^p ds,
\end{align}
and for $\Gamma \subset \partial \Omega$:
\begin{align}
\norm{\nabla^\alpha  u_h -  \nabla^\alpha w_h}_{L^p(\Gamma)}^p
\lesssim 
\sum_{e\in E^{i}_h} h_e^{-p|\alpha|} \int_e |\jumpop{u_h}|^p.
\label{eq:Gammabnd}
\end{align}
\end{Lemma}
\begin{proof}
    We proceed similarly to the proof of  \cite{karakashian2003posteriori}
    adopting partially their notation. For $K\in T_h$ 
    the set of the Lagrange nodes of $K$ is denoted by $\mathcal{N}_K= \{x_{K,j}: j=1, ..., m \}$  
    and  $\{\phi_K^{(x_{K,j})}:\,j=1,..,m\}$ the corresponding basis functions. In addition to this, we set 
    $\mathcal{N}:=\cup_{K\in\mathcal{T}_h}\mathcal{N}_K$. 
    For each $\nu\in \mathcal{N}$, 
    let $\omega_\nu:=\{K\in T_h:\,\nu\in K\}$ and $|\omega_\nu|$ its cardinality. Now, we write 
    \[
    u_h(x):=\sum_{\nu\in \mathcal{N}}\sum_{K\in \omega_\nu}a_K^{(\nu)}\phi_K^{(\nu)}(x)
    \]
    and at every node we define the continuous $w_h$ to be the average of the values of $u_h$, i.e. 
  \begin{align*}
	w_h=\sum_{\nu\in\mathcal{N}}\beta^{(\nu)}\phi^{(\nu)},\,\,\,\text{where}\,\,\,	\beta^{(\nu)}=
				\frac{1}{|\omega_\nu|}\sum_{K\in\omega_\nu} a_K^{(\nu)},  \,\,\, \nu\in \mathcal{N}       ,
		\end{align*}
  and $\phi^{(\nu)}|_K=\phi_K^{(\nu)}$. 
  After a simple scaling argument we know that $\|\nabla^{\alpha} \phi_K^{(\nu)}\|^p_{L^p(K)}\lesssim h_K^{d-p|\alpha|}$. 
  Hence, employing this upper bound and Jensen's inequality we obtain
\begin{align*}
  &\sum_{K \in T_h}\int_K|\nabla^\alpha u_h - \nabla^\alpha w_h|^p dx
   =\sum_{K \in T_h}\int_K\big|\sum_{\nu \in \mathcal{N}_K}(\alpha_K^{(\nu)}-\beta^{(\nu)})
   \nabla^\alpha\phi_K^{(\nu)}(x)\big|^p dx
   \\
  &\lesssim  \sum_{K \in T_h} \sum_{\nu \in \mathcal{N}_K}|\alpha_K^{(\nu)}-\beta^{(\nu)}|^p \int_K|\nabla^\alpha\phi_K^{(\nu)}(x)|^{p} dx 
  \lesssim  \sum_{K \in T_h} \sum_{\nu \in \mathcal{N}_K}
  h_K^{d-p|\alpha|} |\alpha_K^{(\nu)}-\beta^{(\nu)}|^p  
  \\
  &
  = \sum_{\nu \in \mathcal{N}} \sum_{K \in \omega_\nu} h_K^{d-p|\alpha|} |\alpha_K^{(\nu)}-\beta^{(\nu)}|^p  
  \lesssim \sum_{\nu \in \mathcal{N}} h_\nu^{d-p|\alpha|} \sum_{K \in \omega_\nu} 
  |\alpha_K^{(\nu)}-\beta^{(\nu)}|^p,
  \end{align*}
  where $h_\nu=\max_{K\in\omega_\nu}h_K$. 
  Following the same strategy as in \cite{karakashian2003posteriori} for the ordering of $\omega_\nu$ and applying instead Lemma \ref{lemma:tekar} 
  we infer that
  \begin{align*}
      \sum_{K \in T_h}\norm{\nabla^\alpha u_h - \nabla^\alpha w_h}_{L^p(K)}^p 
      \lesssim \sum_{e\in E^i_h}h_e^{d-p|\alpha|}|\jumpop{u_h}|^p_{L^\infty(e)} 
      \lesssim \sum_{e\in E^i_h}h_e^{1-p|\alpha|}|\jumpop{u_h}|^p_{L^p(e)},
  \end{align*}
  where for the last estimate an inverse inequality has been used. 
  Now for (\ref{eq:Gammabnd}) first notice that from the scaling argument one obtains $\|\nabla^{\alpha} 
  \phi_K^{(\nu)}\|^p_{L^p(e)}\lesssim h_e^{d-1-p|\alpha|}$. Defining the set $K_e = \{K\in T_h : e \in K\}$
  and proceeding as before 
  \begin{align*}
      &\sum_{e \in E^b_h \cap \Gamma} \int_e |\nabla^\alpha u_h - \nabla^\alpha w_h|^p dx=\sum_{{e \in E^b_h \cap \Gamma}}\int_e\big|\sum_{\nu \in \mathcal{N}_{K_e}}(\alpha_{K_e}^{(\nu)}-\beta^{(\nu)})
   \nabla^\alpha\phi_{K_e}^{(\nu)}(x)\big|^p ds
   \\
  & \lesssim \sum_{{e \in E^b_h \cap \Gamma}} h_e^{d-1-p|\alpha|} \sum_{\nu \in \mathcal{N}_{K_e}}|\alpha_K^{(\nu)}-\beta^{(\nu)}|^p
  \lesssim  \sum_{{\nu \in \mathcal{N} \cap \Gamma}} h_\nu^{d-1-p|\alpha|} \sum_{K \in \omega_\nu}|\alpha_K^{(\nu)}-\beta^{(\nu)}|^p,
  \end{align*}
  where $h_\nu=\max_{e  \in E^b_h \cap \Gamma}h_e$. Following the same steps as before 
  \begin{align*}
      \sum_{e \in E^b_h \cap \Gamma} \int_e |\nabla^\alpha u_h - \nabla^\alpha w_h|^p dx
      \lesssim \sum_{e\in E^i_h}h_e^{d-1-p|\alpha|}|\jumpop{u_h}|^p_{L^\infty(e)} 
      \lesssim \sum_{e\in E^i_h}h_e^{-p|\alpha|}|\jumpop{u_h}|^p_{L^p(e)}.
  \end{align*}
  
\end{proof}
From Lemma~\ref{lemma:karap} it is straightforward to prove Poincar\'{e} type inequalities for DG spaces. Similar results have been presented in \cite{buffa2009compact}. 

\begin{Theorem}[Poincar\'{e} inequality for DG spaces.]
\label{thm:Poincare}
Let $u_h \in V_h(\Omega)$ and $\Gamma \subset \partial \Omega$ with $|\Gamma| >0$, then for all $h \le 1$ holds
\begin{align}
    \norm{u_h}_{L^p(\Omega)} \lesssim |u_h|_{W^{1,p}(\Omega)} + \norm{u_h}_{L^p(\Gamma)}.
    \label{eq:Poincare}
\end{align}
\end{Theorem}
\begin{proof}
The proof can be found in e.g. \cite{GMKV2023}.
\end{proof}

\section{Convergence of Discrete  Minimisers}\label{Se:Convergence}
In this section we introduce our discontinuous Galerkin method and we show that, in the absence of convexity, our discrete energies $\Gamma$-converge to the relaxed problem. Similarly with \cite{GMKV2023}, a penalty term ensures the discrete stability, but in this case it is slightly modified to include also energies which are not convex.  

For $\Omega\subset \R^d$ a (Lipschitz) bounded domain, we consider the minimisation problem associated with the energy
\begin{equation}\label{eq:energy}
\mathcal{E}(u) = \int_\Omega W(\nabla u(x))\,dx,\quad\text{where}\,\,\, u:\Omega\to \R^N,\,\,u = u_0,\mbox{ on }\partial\Omega.
\end{equation}
We assume throughout that $W:\R^{N\times d}\to\R$ is of class $C^1$  and satisfies the following growth, coercivity, and continuity assumptions
\begin{align}
    -1 + |\xi|^p &\lesssim W(\xi) \lesssim 1 + |\xi|^p \label{eq:growth}\\
   | W(\xi_1) - W(\xi_2) | &\lesssim \left( 1 + |\xi_1|^{p-1} +|\xi_1|^{p-1} \right) | \xi_1 - \xi_2 |, \label{eq:growth2} 
\end{align}
for some $p>1$. Hence, we naturally pose the problem of minimising \eqref{eq:energy} on the Sobolev space 
\[
\mathbb{A}(\Omega) = \left\{u \in W^{1,p}(\Omega)^N\,:\, u = u_0\,\,\mathcal{H}^{d-1}\mbox{ a.e. on $\partial\Omega$}\right\}.
\label{eq:cont_fun_space}
\]
Due to the lack of any convexity assumption on $W$ and hence the possible absence of minimisers of \eqref{eq:energy}, we introduce the relaxed energy
\begin{align}\label{eq:energy_rel}
    \mathcal{E^{\rm{qc}}}(u) = \int_\Omega W^{qc}(\nabla u(x))\,dx, 
\end{align}
where $W^{qc}:\mathbb{R}^{N\times d}\to\mathbb{R}$ is the quasiconvex envelope of $W$, and it is defined as
\begin{align*}
    W^{qc}(\xi):=\sup\{g(\xi):\,g\leq W\,\,\,\text{and}\,\,\,g\,\,\text{quasiconvex}\}.
\end{align*}
It is known that $W^{qc}$ satisfies \eqref{eq:growth} and it is quasiconvex and hence the minimisation problem 
\begin{align*}
    \min_{u\in\mathbb{A}(\Omega)} \mathcal{E^{\rm{qc}}}(u)
\end{align*}
admits minimisers.


Employing the presented ideas in \cite{GMKV2023} for the discretisation of energy functionals in the DG setting  and applying Dirichlet boundary conditions through Nitsche's 
approach, for $u_h\in V^q_h$  our discrete energy functional has the form
\begin{equation}\label{eq:energy_discrete}
\begin{aligned}
\mathcal{E}_h(u_h) = \sum_{K \in T_h}\int_{K}  W(\nabla u_h(x)) -\sum_{e \in E^i_h}\int_{e}\ S  (\Av{ \nabla u_h}  )  
\cdot \J {u_h \otimes n_e}  + \alpha \, {\rm Pen}(u_h),
\end{aligned}
\end{equation}
where $\alpha$ will be chosen large enough and the penalty ${\rm Pen}(u_h)$ is defined as 
\begin{equation}\label{eq:penalty}
\begin{aligned}
{\rm Pen}(u_h)  := & \left(1 + \sum_{K\in T_h} \int_KW(\nabla u_h) + \sum_{e\in E_h} \frac{1}{h_e^{p-1}}\int_e |\jumpop{u_h}|^p \right)^{\frac{p-1}{p}} \times \\
&\qquad\left(\sum_{e\in E_h} \frac{1}{h_e^{p-1}}\int_e |\jumpop{u_h}|^p\right)^{\frac1p}.
\end{aligned}
\end{equation}
where for the boundary faces we use the notation
\begin{align}
\jumpop{u_h} =  u_h^- -u_0, \text{ for } e \in E_h^b = E_h \cap \partial \Omega.
\end{align}
Recall that $u_0$ are the boundary conditions encoded in the continuous function space $\mathbb{A}(\Omega)$ of eq.~(\ref{eq:cont_fun_space}).

\begin{Remark}
\label{rem:penalty}
Alternatively, one may use the penalty term
\begin{equation}\label{eq:penalty1}
\begin{aligned}
{\rm Pen}(u_h) & := \left(1 + |u_h|^{p}_{W^{1,p}(\Omega,T_h)} \right)^{\frac{p-1}{p}} \left(\sum_{e\in E_h} \frac{1}{h_e^{p-1}}\int_e |\jumpop{u_h}|^p \right)^{\frac1p},
\end{aligned}
\end{equation}
We remark that for the purposes of the proofs we will use the penalty term \eqref{eq:penalty1} to avoid lengthy expressions. Note that the proofs remain unaffected as, due to the growth and coercivity conditions \eqref{eq:growth},
\[
\begin{aligned}
1 + |u_h|^{p}_{W^{1,p}(\Omega,T_h)} & \lesssim 1 + \sum_{K\in T_h} \int_KW(\nabla u_h) + \sum_{e\in E_h} \frac{1}{h_e^{p-1}}\int_e |\jumpop{u_h}|^p \\
& \lesssim 1 + |u_h|^{p}_{W^{1,p}(\Omega,T_h)}.
\end{aligned}
\]
\end{Remark}

\subsection{$\Gamma$-convergence of the discrete energies}
We may now state our main result:

\begin{Theorem}
\label{thm:main}
Let $W:\R^{N\times d}\to\R$ be of class $C^1$, satisfying assumptions \eqref{eq:growth} and \eqref{eq:growth2}. Suppose that $u_h \in V^q_h(\Omega)$ is a sequence of minimisers of $\mathcal{E}_h$, i.e.
\[
\mathcal{E}_h(u_h) = \min_{V^q_h(\Omega)} \mathcal{E}_h.
\]
Then, there exists $u\in \mathbb{A}(\Omega)$ such that, up to a subsequence,
\[
u_h \to u\mbox{ in }L^p(\Omega)
\]
and
\[
\mathcal{E}^{\rm qc}(u) = \min_{\mathbb{A}(\Omega)} \mathcal{E}^{\rm qc}.
\]
\end{Theorem}

The proof of Theorem \ref{thm:main} is a direct consequence of the fact that $\mathcal{E}_h$ $\Gamma$-converges to $\mathcal{E}^{\rm qc}$ with respect to the strong topology of $L^p(\Omega)$. The remainder of this section is devoted to the proof of the $\Gamma$-convergence result. We start by proving compactness in an appropriate topology for sequences of bounded energy. 
%

\begin{Lemma}
\label{lemma:compactness}
For $u_h \in V^q_h(\Omega)$ it holds that
\[
\|u_h\|^p_{W^{1,p}(\Omega,T_h)} \lesssim 1 + \mathcal{E}_h(u_h).
\]
In particular, if 
\[
\sup_h\mathcal{E}_h(u_h) \leq C
\]
there exists $u\in \mathbb{A}(\Omega)$ such that, up to a subsequence,
\[
u_h \to u\mbox{ in }L^p(\Omega).
\]
\end{Lemma}

\begin{proof}
For $u_h \in V^q_h(\Omega)$, using assumption \eqref{eq:growth2}, we estimate the second term in the discrete energy \eqref{eq:energy_discrete} as follows:
\begin{align*}
    &\int_{\Gamma_{\text{int}}}\big | DW(\{\{\nabla_h u_h\}\}): \J {v_h \otimes n_e} \big |ds=\int_{\Gamma_{\text{int}}}\big |h^{\frac{p-1}{p}}DW(\{\{\nabla_h u_h\}\}): h^{-\frac{p-1}{p}}\J {v_h \otimes n_e}\big | \\
&\lesssim \left(\int_{\Gamma_{\text{int}}}h(1+|\{\{\nabla_h u_h\}\}|^{p-1})^{\frac{p}{p-1}}ds\right)^{\frac{p-1}{p}}\left(\int_{\Gamma_{\text{int}}}h^{1-p}|\J {v_h \otimes n_e}|^p ds\right)^{1/p}\\
    & \lesssim \left(\int_{\Gamma_{\text{int}}}h(1+|\{\{\nabla _h u_h\}\}|^{p})ds\right)^{\frac{p-1}{p}}\left(\int_{\Gamma_{\text{int}}}h^{1-p}|\J {v_h \otimes n_e}|^p ds\right)^{1/p},
\end{align*}
where for the first term on the RHS we further use an inverse inequality to infer that 
\begin{align*}
    \int_{\Gamma_{\text{int}}}h(1+|\{\{\nabla_h u_h\}\}|^{p})ds&\lesssim \int_{\Gamma_{\text{int}}}h(1+|\nabla_h u_h^+|^{p}+|\nabla_h u_h^-|^{p})ds\\
    &\lesssim \sum_{K\in\mathcal{T}_h}\int_{\partial K}h(1+|\nabla_h u_h|^{p})ds\\
    &\lesssim \sum_{K\in\mathcal{T}_h}\int_{ K}1+|\nabla_h u_h|^{p}dx.
\end{align*}

\noindent Hence, by Young's inequality, we infer that 
\begin{align}\label{eq:compact_1}
\int_{\Gamma_{\text{int}}}\big |DW(\{\{\nabla_h u_h\}\}): \J {v_h \otimes n_e}\big |  & \lesssim  \delta^{1-p} \sum_{e \in E_h^i} h_e^{1-p} \int_e |\jumpop{u_h}|^p \nonumber \\
&\quad + \delta  \int_\Omega \left(1 + |\nabla_h u_h|^{p}\right).
\end{align}
The coercivity assumption on $W$ says that
\[
W(\xi) \gtrsim -1 + |\xi|^p.
\]
Thus, choosing $\delta$ appropriately as a function of $p$, by \eqref{eq:compact_1} we find that
\[
\mathcal{E}_h(u_h) \gtrsim  \int_\Omega |\nabla_h u_h|^p - 1 - C(p)\sum_{e \in E_h^i} \frac{1}{h_e^{p-1}} \int_e |\jumpop{u_h}|^p + \alpha\,{\rm Pen}(u_h).
\] 
Next, note that since 
\begin{equation}\label{eq:norm_ineq}
 |u_h|^{p}_{W^{1,p}(\Omega,T_h)} \geq  \sum_{e\in E^i_h} \frac{1}{h_e^{p-1}}\int_e |\jumpop{u_h}|^p
\end{equation}
the penalty term in \eqref{eq:penalty} satisfies
\begin{align*}
{\rm Pen}(u_h) & = \left(1 + |u_h|^{p}_{W^{1,p}(\Omega,T_h)} \right)^{\frac{p-1}{p}} \left(\sum_{e\in E_h} \frac{1}{h_e^{p-1}}\int_e |\jumpop{u_h}|^p \right)^{\frac1p}\\
&\geq \left(1 + \sum_{e\in E^i_h} \frac{1}{h_e^{p-1}}\int_e |\jumpop{u_h}|^p \right)^{\frac{p-1}{p}} \left(\sum_{e\in E^i_h} \frac{1}{h_e^{p-1}}\int_e |\jumpop{u_h}|^p \right)^{\frac1p} \\
& \geq \sum_{e\in E^i_h} \frac{1}{h_e^{p-1}}\int_e |\jumpop{u_h}|^p,
\end{align*}
where the last inequality follows from the fact that the conjugate exponents $(p-1)/p$ and $1/p$ add up to $1$. Then, provided that $\alpha > C(p)$, we deduce that
\[
\mathcal{E}_h(u_h) \gtrsim  \int_\Omega |\nabla_h u_h|^p - 1 + (\alpha - C(p))\sum_{e\in E_h^i} \frac{1}{h_e^{p-1}} \int_e |\jumpop{u_h}|^p
\] 
and hence that
\[
|u_h|^p_{W^{1,p}(\Omega,T_h)} \lesssim 1 + \mathcal{E}_h(u_h).
\]
Furthermore, if $\mathcal{E}_h(u_h) \leq C$, $||u_h||_{L^p(\partial \Omega)}$ is uniformly bounded which implies from the Poincar\'{e} inequality~(\ref{eq:Poincare}) that $||u_h||_{L^p(\Omega)} \leq C^\prime$. From the compactness of the space $W^{1,p}(\Omega,T_h)$, see \cite[Theorem 5.2]{buffa2009compact},
we infer that there exists a function $u\in W^{1,p}(\Omega)$ such that 
\begin{align*}
    u_h\to u\,\,\,\text{in}\,\,L^p(\Omega).
\end{align*}
To prove $u \in \mathbb{A}(\Omega)$ note that
\[
||u - u_0||_{L^p(\partial \Omega)} \le ||u_h - u_0||_{L^p(\partial \Omega)} + ||u - u_h||_{L^p(\partial \Omega)}.
\]
As $h\rightarrow 0$  the first term of the right hand side converge to zero due to the assumption that $\sup_h\mathcal{E}_h(u_h) \leq C$ and
$||u - u_h||_{L^p(\partial \Omega)} \rightarrow 0$  from the compact embedding in \cite[Lemma 8]{buffa2009compact}.
\end{proof}

We next prove the $\liminf$-inequality, providing a lower bound for the sequence $\mathcal{E}_h(u_h)$. 

\begin{Lemma}
\label{lemma:liminf}
Suppose that $u_h \in V^q_h(\Omega)$ and $u\in \mathbb{A}(\Omega)$ satisfy $u_h \to u$ in $L^p(\Omega)$. Then,
\[
\liminf_{h\to 0} \mathcal{E}_h(u_h) \geq \mathcal{E}^{\rm qc}(u).
\]
\end{Lemma}

\begin{proof}
Given $u_h$, $u$ as in the statement, suppose that $\liminf \mathcal{E}_h(u_h) \leq C$ as otherwise there is nothing to prove. By passing to a subsequence, we may thus assume that
\[
\sup_h \mathcal{E}_h(u_h) \leq C
\]
and by Lemma \ref{lemma:compactness} we have that $\|u_h\|_{W^{1,p}(\Omega,T_h)} \leq C$ and that
\begin{equation}\label{eq:liminf_1}
u_h\to u\mbox{ in }L^p(\Omega),
\end{equation}
where $u\in W^{1,p}(\Omega)$. Now, from Lemma \ref{lemma:karap}, we extract a sequence $w_h\in V_h(\Omega) \cap W^{1,p}(\Omega)$  such that
\begin{align}
\int_\Omega| u_h -  w_h|^p
&\lesssim \sum_{e\in E^i_h} h_e \int_e |\jumpop{u_h}|^p ds \label{eq:k1},\\
\int_\Omega|\nabla_h u_h - \nabla w_h|^p
&\lesssim \sum_{e\in E^i_h} h_e^{1-p} \int_e |\jumpop{u_h}|^p ds, \label{eq:k2}
\end{align}
and so combined with \eqref{eq:liminf_1} we get that $w_h\to u$ in $L^p(\Omega)$. In addition to this, since
\begin{align}\label{eq:k3}
    \int_\Omega|\nabla w_h|^p\lesssim \sum_{e\in E^i_h} \frac{1}{h_e^{p-1}}\int_e |\jumpop{u_h}|^p+\int_\Omega|\nabla_h u_h|^p\lesssim |u_h|^p_{W^{1,p}(\Omega,\mathcal{T}_h)},
\end{align}
we infer that $(\nabla w_h)$ is bounded in $L^p(\Omega)$ and hence there exists $G\in L^p(\Omega)$ such that $\nabla w_h\rightharpoonup G$ in $L^p(\Omega)$. However, the strong convergence of $w_h$ in $L^p(\Omega)$ and the distributional continuity of the operator $\nabla: L^p(\Omega)\to L^p(\Omega)$, give us that $G=\nabla u$ and so
\begin{align*}
    \nabla w_h\rightharpoonup\nabla u,\,\,\,\text{in}\,\, L^p(\Omega).
\end{align*}
Note that since $W^{qc}$ is quasiconvex and $W\geq W^{qc}$, it holds that
\begin{equation}\label{eq:liminf_2}
\liminf_{h\to 0}\int_\Omega W(\nabla w_h) \geq\liminf_{h\to 0}\int_\Omega W^{qc}(\nabla w_h)\geq \int_\Omega W^{qc}(\nabla u) =\mathcal{E}^{\rm qc}(u).
\end{equation}
We thus add and subtract the term
\[
\int_\Omega W(\nabla w_h)
\]
in the discrete energy $\mathcal{E}_h$ to find that
\[
\begin{aligned}
\mathcal{E}_h(u_h) & = \int_\Omega W(\nabla w_h) + \int_\Omega W(\nabla_h u_h) - W(\nabla w_h) \\
&\qquad - \int_{\Gamma_{\text{int}}}DW(\{\{\nabla_h u_h\}\}): \J {v_h \otimes n_e} +  \alpha\, {\rm Pen}(u_h) \\
&\qquad =: I_h + II_h + III_h +  \alpha\, {\rm Pen}(u_h).
\end{aligned}
\]
By \eqref{eq:liminf_2} it suffices to show that
\begin{equation}\label{eq:sufficient}
\liminf_{h\to 0} \left[ II_h + III_h +  \alpha\, {\rm Pen}(u_h) \right] \geq 0. 
\end{equation}
Regarding the term $III_h$ we argue similarly to Lemma \ref{lemma:compactness} to get that
\begin{equation*}
| III_h |  \lesssim \left(\sum_{K\in\mathcal{T}_h}\int_{ K}1+|\nabla_h u_h|^{p}dx\right)^{1/p'}\left(\int_{\Gamma_{\text{int}}}h^{1-p}|\J {v_h \otimes n_e}|^p ds\right)^{1/p},
\end{equation*}
which in turns implies that
\begin{equation}\label{eq:liminf6}
\begin{aligned}
| III_h | &\leq C \left(1 + |u_h |_{W^{1,p}(\Omega,T_h)}^{p}\right)^{\frac{p-1}{p}}\left(\sum_{e\in E^i_h} \frac{1}{h_e^{p-1}}\int_e |\jumpop{u_h}|^p\right)^{\frac1p} \\
&\leq C\, {\rm Pen}(u_h).
\end{aligned}
\end{equation}

Next, we estimate the term $II_h$. By H\"older's inequality, \eqref{eq:k2}, \eqref{eq:k3} and \eqref{eq:growth2} we find that
\begin{equation}\label{eq:liminf7}
\begin{aligned}
| II_h | & \lesssim \int_\Omega \left( 1 + |\nabla_h u_h|^{p-1} +|\nabla w_h|^{p-1} \right) | \nabla_h u_h - \nabla w_h | \\
& \lesssim  \left[\int_\Omega \left( 1 + |\nabla_h u_h|^{p-1} +|\nabla w_h|^{p-1} \right)^{\frac{p}{p-1}} \right]^{\frac{p-1}{p}} \left(\int_\Omega|\nabla_h u_h-\nabla w_h|^p\right)^\frac{1}{p} \\
& \lesssim \left[\int_\Omega \left( 1 + |\nabla_h u_h|^{p} +|\nabla w_h|^{p} \right) \right]^{\frac{p-1}{p}}  \left(\sum_{e\in E^i_h} \frac{1}{h_e^{p-1}}\int_e |\jumpop{u_h}|^p\right)^{\frac1p}\\
&\lesssim\left( 1 + |u_h|_{W^{1,p}(\Omega,T_h)}^p  \right)^{\frac{p-1}{p}}  \left(\sum_{e\in E^i_h} \frac{1}{h_e^{p-1}}\int_e |\jumpop{u_h}|^p\right)^{\frac1p}\\
&\leq \tilde C \,{\rm Pen}(u_h).
\end{aligned}
\end{equation}
Combining the latter estimate with \eqref{eq:liminf6}, we get that 
\[
II_h + III_h +  \alpha\, {\rm Pen}(u_h) \geq \left(\alpha - C -\tilde C\right) {\rm Pen}(u_h) 
\]
and choosing $\alpha > C+\tilde C $ we infer \eqref{eq:sufficient} which concludes the proof.
\end{proof}

The remaining ingredient in the proof of Theorem \ref{thm:main} is providing a recovery sequence which is given in the following lemma.

\begin{Lemma}
\label{lemma:limsup}
For every $u \in \mathbb{A}(\Omega)$ if $q \ge 1$ and $h<1$ there exist $u_h \in V^q_h(\Omega)$  such that $u_h \to u$ in $L^p(\Omega)$ and 
\[
\limsup_{h\to 0} \mathcal{E}_h(u_h) \leq \mathcal{E}^{\rm qc}(u).
\]
\end{Lemma}

\begin{proof}
First, similarly to \cite{GMKV2023}[Lemma 3.6],  we show that for all $v \in \mathbb{A}(\Omega)$ there exists a recovery sequence $(u_h) \subset V_h^q$ such that 
\begin{align}
|| v - u_h ||^p_{W^{1,p}(\Omega,T_h)} +  \sum_{e\in E^b_h}\frac{1}{h_e^{p-1}}\int_e |u_h - u_{0}|^p \rightarrow 0, 
\text{ as } h \rightarrow 0.
\label{limsup_uh_conv}
\end{align}
and 
\begin{align}
\lim_{h\rightarrow 0} \mathcal{E}_h(u_h) = \mathcal{E}(v).
\label{eq:limsup_Eh_conv}
\end{align}
\\
First note that $C^\infty(\bar{\Omega})$ is dense in $W^{1,p}(\Omega)$ and one can create a sequence 
$(u_\delta) \subset C^\infty(\bar{\Omega})$ with the properties
\begin{align}
||v - u_\delta||_{W^{1,p}(\Omega)} \lesssim \delta \quad \text{and} \quad |u_\delta|_{W^{2,p}(\Omega)} \lesssim \frac{1}{\delta} |v|_{W^{1,p}(\Omega)}.
\label{eq:mol_ineq}
\end{align}
Choosing the sequence $u_{h, \delta} = I_h^q u_\delta$, $I_h^q: W^{s,p}(\Omega) \rightarrow V_h^q(\Omega)$ is the standard nodal interpolation operation,
we recall the error estimates for all $q \ge 1$ 
\begin{align}
&|u_\delta - u_{h,\delta}|_{W^{i,p}(K)} \lesssim h_K^{2-i} |u_\delta |_{W^{2,p}(K)}, 
\label{eq:error_in_K} \\ 
&|u_\delta - u_{h,\delta}|_{L^p(e)} \lesssim h_K^{2-\frac{1}{p} }|u_\delta|_{W^{2,p}(K)}.
\label{eq:error_in_e} 
\end{align}
In the following, assuming that $h <1$, the above error estimates are combined with eq.~(\ref{eq:mol_ineq}) 
\begin{align}
&\sum_{K \in T_h} |u_\delta - u_{h, \delta}|_{W^{i,p}(K)}^p \lesssim h^p |u_\delta |_{W^{2,p}(\Omega)}^p \lesssim \frac{h^p}{\delta^p} |v|_{W^{1,p}(\Omega)}^p, 
\label{eq:bulk_terms}
\\
&\sum_{e \in E_h} \frac{1}{h_e^{p-1}} || \jumpop{u_\delta - u_{h, \delta}}||_{L^p(e)}^p \lesssim h^p \sum_{e \in E_h}|u_\delta|^p_{W^{2,p}(K_e)} 
\lesssim \frac{h^p}{\delta^p} |v|_{W^{1,p}(\Omega)}^p,
\label{eq:jump_terms}
\end{align}
here $K_e$ denotes the set of the elements containing   $e$, $e \in E_h$. For the boundary terms applying Dirichlet boundary conditions through Nitsche's approach, 
the above estimates together with eq.~(\ref{eq:mol_ineq}) imply
\begin{equation}
\begin{aligned}
&\sum_{e\in E^b_h}\frac{1}{h_e^{p-1}}\int_e |u_{h,\delta} - u_0|^p  \lesssim \sum_{e\in E^b_h}\frac{1}{h_e^{p-1}}\int_e \left( |u_{h, \delta} - u_{\delta}|^p + |u_\delta - v|^p \right) 
\\
&\lesssim \left( \frac{h^{p}}{\delta^{p}} + \frac{\delta^p}{h^{p-1}}\right) || v||^p_{W^{1,p}(\Omega)}.
\label{eq:bnd_limsup}
\end{aligned}
\end{equation}
For $\beta \in (\frac{p-1}{p}, 1)$, choosing $\delta = h^\beta$ equations~(\ref{eq:bulk_terms}), (\ref{eq:jump_terms}) and (\ref{eq:bnd_limsup})
$\rightarrow 0$ as $h\rightarrow 0$. Here $v\in W^{1,p}(\Omega)$ ($u_\delta$ is continuous by construction), which means $\jumpop{v} =0$ on every $e \in E_h^i$ and as $h\rightarrow 0$
\begin{align}
    \norm{v -u_{h,\delta}}^p_{W^{1,p}(\Omega, T_h)} \le \norm{v -u_\delta}_{W^{1,p}(\Omega)}
    + \norm{u_\delta - u_{h,\delta}}_{W^{1,p}(\Omega, T_h)}^p \rightarrow 0.
\end{align}
Therefore, we found a sequence in $V_h^q(\Omega)$ such that eq.~\eqref{limsup_uh_conv} holds.


%

Next, for the convergence of the discrete functional,
eq.~(\ref{eq:limsup_Eh_conv}) and considering again the sequence $(w_h)$ of Lemma \ref{lemma:karap}, we note that
\begin{equation}\label{eq:limsup_1}
\nabla w_h \to \nabla v \mbox{ in }L^p(\Omega). 
\end{equation}
Indeed, recalling that  $\jumpop{v} = 0$, we find that
\begin{align*}
\int_\Omega |\nabla w_h - \nabla v|^p & \lesssim \int_\Omega |\nabla_h u_h - \nabla v|^p + \int_\Omega |\nabla_h u_h-\nabla w_h|^p\\
& \lesssim  \int_\Omega |\nabla_h u_h - \nabla v|^p + \sum_{e\in E_h^i} \frac{1}{h_e^{p-1}} \int_e |\jumpop{u_h} |^p\\
& =  \sum_K \int_K |\nabla_h u_h - \nabla v|^p + \sum_{e\in E_h^i} \frac{1}{h_e^{p-1}} \int_e |\jumpop{u_h - v} |^p\\
& = |u_h - v|^p_{W^{1,p}(\Omega,T_h)}\to 0.
\end{align*}

As in Lemma \ref{lemma:liminf}, we can write the discrete energy as
\[
\begin{aligned}
\mathcal{E}_h(u_h) & = \int_\Omega W(\nabla w_h) + \int_\Omega W(\nabla_h u_h) - W(\nabla w_h) \\
&\qquad -\int_{\Gamma_{\text{int}}}DW(\{\{\nabla_h u_h\}\}): \J {u_h \otimes n_e} +  \alpha\, {\rm Pen}(u_h)  \\
&\qquad =: I_h + II_h + III_h +  \alpha\, {\rm Pen}(u_h).
\end{aligned}
\]
Vitali's convergence theorem, \eqref{eq:limsup_1}, and the growth of $W$ say that
\begin{equation}\label{eq:limsup_2}
I_h = \int_\Omega W(\nabla w_h) \to \int_\Omega W(\nabla v) = \mathcal{E}(v) \mbox{ as } h\to 0,
\end{equation}
and since the boundary terms converge to zero, eq.~(\ref{limsup_uh_conv}),  it remains to show that $II_h$, $III_h$ 
as well as ${\rm Pen}(u_h)$ all vanish in the limit $h\to 0$. Using the estimates \eqref{eq:liminf6} and \eqref{eq:liminf7}, we find that 
\[
 |II_h| + |III_h| \lesssim  {\rm Pen}(u_h) 
\]
and it thus suffices to prove that 
\[
\limsup_{h\to0}{\rm Pen}(u_h)  = 0.
\]
Indeed, since $ |u_h - v |_{W^{1,p}(\Omega,T_h)} \to 0$, it follows that
\[
|u_h|_{W^{1,p}(\Omega,T_h)} \leq C
\]
and hence
\begin{align*}
{\rm Pen}(u_h) & \lesssim \left( \sum_{e \in E_h} \frac{1}{h_e^{p-1}} \int_e |\jumpop{u_h} |^p\right)^{\frac1p}  
 \lesssim \left( \sum_{e\in E_h} \frac{1}{h_e^{p-1}} \int_e |\jumpop{u_h - v} |^p\right)^{\frac1p} \\
 & \lesssim \left( |u_h -v|_{W^{1,p}(\Omega,T_h) } + 
 \sum_{e\in E_h^b} \frac{1}{h_e^{p-1}} \int_e |\jumpop{u_h - v} |^p \right)^{\frac1p} \to 0 .
\end{align*}
This completes the proof of \eqref{eq:limsup_Eh_conv}.

Let now $u\in\mathbb{A}(\Omega)$, then from \cite{dacorogna2007direct}[Theorem 9.1] the exists a sequence $(u_k)\subset \mathbb{A}(\Omega)$ such that 
\begin{align*}
    u_k\to u,\,\,\, \text{in}\,\, L^p(\Omega)\quad\text{and}\quad
    \mathcal{E}_k(u_k)\to \mathcal{E}^{\rm qc}(u),
\end{align*}
for $k\to\infty$. Using the first part of our proof, for each $k\in\mathbb{N}$, we can further extract a sequence $u_{k,h}\subset V_h^q$ such that 
\begin{align*}
    u_{k,h}\to u_k,\,\,\, \text{in}\,\, L^p(\Omega),\quad\text{and}\quad
    \mathcal{E}_h(u_{k,h})\to \mathcal{E}(u_k),
\end{align*}
for $h\to 0$. Hence, combining the above convergences we have that
\begin{align*}
    \lim_{k\to\infty}\lim_{h\to 0}\mathcal{E}_h(u_{k,h})=\lim_{k\to\infty}\mathcal{E}(u_k)=\mathcal{E}^{\rm qc}(u),
\end{align*}
and 
\begin{align*}
\lim_{k\to\infty}\lim_{h\to 0}\|u_{k,h}-u\|_{L^p(\Omega)}=0.
\end{align*}
The latter implies that we can choose a subsequence $(k_h)$ such that for $u_{k_h,h}\subset V_h^q$ it holds that
\begin{align*}
     &\lim_{h\to 0}\mathcal{E}_{h}(u_{k_h,h})=\mathcal{E}^{\rm qc}(u),\,\,\,\text{and},\\
     &\lim_{h\to 0}\|u_{k_h,h}-u\|_{L^p(\Omega)}=0,
\end{align*}
which concludes the proof of our Lemma.
\end{proof}

We conclude by demonstrating the proof of Theorem \ref{thm:main}.

\begin{proof}[Proof of Theorem \ref{thm:main}]

Suppose that $u_h \in V^q_h$ is given such that
\[
\mathcal{E}_h(u_h) = \inf_{V^q_h} \mathcal{E}_h.
\]
Note that for any $v \in \mathbb{A}(\Omega)$, by Lemma \ref{lemma:limsup} we can find a sequence $v_h \in V^q_h$ such that
\[
v_h \to v \mbox{ in }L^p(\Omega)
\]
and $\limsup \mathcal{E}_h(v_h) \leq \mathcal{E}^{\rm qc}(v)$. In particular, up to a subsequence,
\[
\mathcal{E}_h(u_h) \leq \mathcal{E}_h(v_h) \leq 1 + \limsup \mathcal{E}_h(v_h) \leq 1 + \mathcal{E}^{\rm qc}(v) < \infty.
\]
Hence, by Lemma \ref{lemma:compactness}, there exists $u\in \mathbb{A}(\Omega)$ such that
\[
u_h \to u \mbox{ in }L^p(\Omega)
\]
and, for any $v \in \mathbb{A}(\Omega)$,
\[
\mathcal{E^{\rm qc}}(u) \leq \liminf_{h\to0} \mathcal{E}_h(u_h) \leq \limsup \mathcal{E}_h(v_h) \leq \mathcal{E}^{\rm qc}(v),
\]
where $v_h$ is the recovery sequence for $v$ whose existence is guaranteed by Lemma \ref{lemma:limsup}. That is, 
\[
\mathcal{E}^{\rm qc}(u) = \inf_{\mathbb{A}(\Omega)} \mathcal{E}^{\rm qc}
\]
completing the proof.
\end{proof}

\subsection{DG methods based on discrete gradients}\label{Se:discrG}
We start with a definition of discrete gradients in the discontinuous Galerkin setting. To this end, 
define first the lifting operator $R_h : T(E_h) \to V^{q-1}_h(\Omega)$  as
\begin{align}
 \int_\Omega R_h(\varphi) : w_h \,dx =\sum_{e\in E^i_h} \int_e \dgal{ w_h } : 
\jumpop{ \varphi \otimes n_e } \, ds, \quad  \forall w_h \in V^{q-1}_h(\Omega).
\label{eq:lift_operator}
\end{align}
This operator leads to the definition of  the discrete gradient $G_h$ for $u_h\in V^q_h$ as
\begin{align}
G_h(u_h) = \nabla_h u_h - R_h(u_h ).
\label{discrete_gradient}
\end{align}
The method of Ten Eyck and Lew (2006) \cite{ten2006discontinuous} is based on minimising over the DG space the functional   defined on $ V^{q}_h:$
\begin{equation}\label{en_dg_lo}
\begin{aligned}
\mathcal{E}_{\text{G}   , h} [u_h] &= \sum_{K \in T_h}\int_{K}  W(\nabla u_h(x)  - R_h(u_h ))  
+ \alpha\text{Pen}\, ( {u_h} )\\
&= \sum_{K \in T_h}\int_{K}  W(G_h(u_h))  
+ \alpha\text{Pen}\, ( {u_h} ).
 \end{aligned}
\end{equation}
Convergence in the case of convex $W$ was derived in \cite{buffa2009compact} and error estimates for smooth solutions were derived in   
\cite{ortner2007discontinuous}.   The evaluation of this functional requires the
 computation of $R_h(u_h ),$ and thus of   the ``discrete gradient"  $G_h(u_h)$ which appears   in $W.$
Computationally, the cost of this operation is reasonable, since it involves the solution of local systems.   
In the convex case,  the weak convergence properties of the discrete gradient were enough to establish lim-inf inequalities and conclude the proof of convergence of minimisers. In our setting this is no longer the case. As in the previous section 
the key idea is still to use  Lemma \ref{lemma:karap} and to  extract a sequence $w_h\in V_h(\Omega) \cap W^{1,p}(\Omega)$  such that
\eqref{eq:k1}, \eqref{eq:k2} will hold.
Then,  \eqref{eq:liminf_1} will imply  that $w_h\to u$ in $L^p(\Omega)$. 
Repeating the same arguments as in Lemma \ref{lemma:liminf}, we conclude
\begin{equation}\label{eq:liminf_2_DGr}
\liminf_{h\to 0}\int_\Omega W(\nabla w_h) \geq\liminf_{h\to 0}\int_\Omega W^{qc}(\nabla w_h)\geq \int_\Omega W^{qc}(\nabla u) =\mathcal{E}^{\rm qc}(u).
\end{equation}
And thus it suffices to show that 
\begin{equation}\label{eq:sufficient_dGr}
\liminf_{h\to 0} \left[ A _h +    \alpha\, {\rm Pen}(u_h) \right] \geq 0. 
\end{equation}
where 
\[
\begin{aligned}
\mathcal{E}_{\text{G}   , h}  (u_h) & = \int_\Omega W(\nabla w_h) + \int_\Omega W(G_h( u_h)) - W(\nabla w_h) +  \alpha\, {\rm Pen}(u_h)  \\
&\qquad =: \int_\Omega W(\nabla w_h)   + A_h  +  \alpha\, {\rm Pen}(u_h).
\end{aligned}
\]
To estimate $A_h$
we shall use the known bound for the lifting operator, \cite{buffa2009compact,di2010discrete},  
  \begin{align}
   \int_\Omega |R_h(v_h)|^p  \leq 
  C_R \sum_{e \in E_h^i}h_e^{1-p} \int_e |\jumpop{v_h}|^p, \qquad v_h \in V^q_h(\Omega), 
  \label{eq:R_bound}
  \end{align}
  where the constant $C_R$ is independent of $h$. A consequence of this bound, see \cite{GMKV2023}, is 
  that for all $v_h \in V^q_h(\Omega)$ it holds that 
\begin{equation}
 \label{eq:R_bound_2}
\|G_h(v_h) \|_{L^p(\Omega)} \leq C |v_h|_{W^{1,p}(\Omega,T_h)}.	
\end{equation}
 Next, proceeding as in Lemma  \ref{lemma:liminf}, we conclude
\begin{equation*}
\begin{aligned}
| A_h | & \lesssim \int_\Omega \left( 1 + |G_h( u_h) |^{p-1} +|\nabla w_h|^{p-1} \right) |G_h( u_h) - \nabla w_h | \\
& \lesssim  \left[\int_\Omega \left( 1 + |G_h( u_h) |^{p-1} +|\nabla w_h|^{p-1} \right)^{\frac{p}{p-1}} \right]^{\frac{p-1}{p}} \left(\int_\Omega|G_h( u_h) -\nabla w_h|^p\right)^\frac{1}{p} \\
& \lesssim \left[\int_\Omega \left( 1 + |G_h( u_h)|^{p} +|\nabla w_h|^{p} \right) \right]^{\frac{p-1}{p}} 
\left [ \left(\int_\Omega|\nabla_h u_h -\nabla w_h|^p\right)^\frac{1}{p} + \left(\int_\Omega|R_h(v_h)|^p\right)^\frac{1}{p} \right ]\\
&\lesssim\left( 1 + |u_h|_{W^{1,p}(\Omega,T_h)}^p  \right)^{\frac{p-1}{p}}  \left(\sum_{e\in E^i_h} \frac{1}{h_e^{p-1}}\int_e |\jumpop{u_h}|^p\right)^{\frac1p}\\
&\leq \tilde C \,{\rm Pen}(u_h).
\end{aligned}
\end{equation*}
As a result,  
\[
A_h  +  \alpha\, {\rm Pen}(u_h) \geq \left(\alpha   -\tilde C\right) {\rm Pen}(u_h) \, .
\]
Thus, if $\alpha >  \tilde C ,$  \eqref{eq:sufficient_dGr} follows. 
The lim-sup inequality follows by adopting in similar fashion arguments from the previous section. 
We therefore obtain the following theorem.

\begin{Theorem}
\label{thm:main_dGr}
Let $W:\R^{N\times d}\to\R$ be of class $C^1$, satisfying assumptions \eqref{eq:growth} and \eqref{eq:growth2}. Suppose that $u_h \in V^q_h(\Omega)$ is a sequence of minimisers of $\mathcal{E}_{\text{G}   , h}$, i.e.
\[
\mathcal{E}_{\text{G}   , h}(u_h) = \min_{V^q_h(\Omega)} \mathcal{E}_{\text{G}   , h}.
\]
Then, there exists $u\in \mathbb{A}(\Omega)$ such that, up to a subsequence,
\[
u_h \to u\mbox{ in }L^p(\Omega)
\]
and
\[
\mathcal{E}^{\rm qc}(u) = \min_{\mathbb{A}(\Omega)} \mathcal{E}^{\rm qc}.
\]
\end{Theorem}


\section{Computational Results}
\label{sec:computations}

In this section it is illustrated how the proposed numerical scheme approximates minimisers for the 
total potential energy when 
polyconvex (and hence quasiconvex) and multi-well strain energy functions are employed. In all computations we use the scheme described in Section 3, see \eqref{eq:energy_discrete}.  For the first case we compare our 
scheme with a generalisation of the standard $L^2$ interior penalty method for convex energies, 
\cite{GMKV2023},  while for the latter case we show how a sequence of discrete minimisers is formed 
which converges to the minimiser of the relaxed problem.

In the sequel $y :\R^2\supset\Omega  \rightarrow \Omega^*$ will denote 
plane deformations from $\Omega$ (undeformed body) to $\Omega^*\subset \R^2$ (deformed configuration). Under given boundary 
conditions we seek deformations that minimise the total potential energy
\begin{align}
\min_{\substack{ y \in W^{1,p}(\Omega) \\ y = y_0 \text{ on } \partial \Omega} } \mathcal{E}( y) =  \min_{\substack{ y \in W^{1,p}(\Omega) \\ y = y_0 \text{ on } \partial \Omega} }
\int_\Omega W(\nabla y) dx,
\label{sim:energMin}
\end{align}
where $W$ is the strain energy function characterising the material properties.

\subsection*{Examples with a polyconvex strain energy function}
We would like to approximate minimisers of eq.~(\ref{sim:energMin}) through the discretised energy (\ref{eq:energy_discrete}) for the polyconvex strain energy function  $W(F) = |\det F|^2$. We compare the 
proposed penalty of eq.~(\ref{eq:penalty}) with an appropriate penalty for convex energies, 
\cite{GMKV2023}, given by
\begin{align}\label{eq:convexpen}
{\rm Pen}(u_h) & := \left(1 + |u_h|^{p-2}_{W^{1,p}(\Omega,T_h)} \right) \left(\sum_{e\in E_h} \frac{1}{h_e^{p-1}}\int_e |\jumpop{u_h}|^p \right)^{\frac2p},
\end{align}
which for $p=2$ reduces to the standard $L^2$ error. In eq.~(\ref{sim:energMin}) we choose $y_0(x) = F_0 x$ where
\begin{align}
    F_0 = \begin{pmatrix}
1 & 0\\
0 & 0.9
\end{pmatrix}.
\end{align}
corresponding to uniaxial compression. Here $y_0$ is a homogeneous deformation,
therefore polyconvexity implies that
\begin{align}\label{eq.polineq}
\int _\Omega W(\nabla y) \ge \int_\Omega W(\nabla y_0)  \text{ for all } y \in W^{1, 4}(\Omega) \text{ with }
y = y_0 \text{ on } \partial \Omega,
\end{align}
which means that every admissible deformation results in a potential energy no  less than that of $y_0$. We seek discrete minimisers in $V_h^1(\Omega)$, i.e. piecewise polynomials of first degree.
The structure of the described minimisation problem indicates that for big enough stabilisation parameter $\alpha$ in 
eq.~(\ref{eq:energy_discrete}), the computed minimiser should satisfy inequality \eqref{eq.polineq}. As it is 
illustrated in fig.~\ref{W_over_alpha} where the blue curve approaches the black dashed line, large values of the stabilisation parameter $\alpha$ are required to approach $\int_\Omega W(\nabla y_0)$ when the penalty of eq.~\eqref{eq:convexpen} is used. 
On the contrary for the penatly of eq.~\eqref{eq:penalty}, even small values of $\alpha$ bring the discrete minimiser close to $W^{1,4}(\Omega)$, i.e. the orange curve is always close to the black dashed line.
To examine further this behaviour we plot $1/\det \nabla y$ in the deformed configuration, 
fig.~\ref{fig:compression_dens}. For both penalties $y_h$ approaches $y_0$ as $\alpha$ increases but for the
term \eqref{eq:convexpen} large values of $\alpha$ are required while for \eqref{eq:penalty} even small
values produce a minimiser close to $y_0$. These observations are quantified in fig.~\ref{fig_errors}
where we compare the $L^1(\Omega)$ and $W^{1,1}(\Omega)$ errors 
between $y_0$ and the computed minimisers $y_h$ varying the penalty parameter $\alpha$.

\subsection*{Non rank-one convex energies and minimising sequences }

We construct a frame indifferent two-well strain energy function where under specific boundary conditions a phase mixture
emerges with finer and finer microstructure, approaching the infimum of the total potential energy as we decrease the 
mesh size $h$. Mimicking the cubic to
orthorhombic phase transitions in crystals we define the energy $W(F) = |C-V^2| |C-I|^2$, where $C = F^T F$ is the right
Cauchy-Green tensor, $I$ is the identity and
\begin{align}
    V = \begin{pmatrix}
\frac{a_o+ b_o}{2} & \frac{b_o -a_o}{2}\\
\frac{b_o - a_o}{2} & \frac{a_o+ b_o}{2}
\end{pmatrix}.
\end{align}
Choosing  $a_0=\sqrt{2 -b_0^2}$ and $ b_0 = 0.9$,
equation $R V - I = d \otimes n$ has two 
solutions namely
$\{R_1, d_1, n_1\}$, $\{R_2, d_2, n_2\}$ such that $n_1 \approx (1,0)$ and $n_2 \approx (0,1)$, where $R_i \in SO(2)$, see \cite[Proposition 4]{ball1989fine}.
This essentially means, that there exist continuous deformations with discontinuous deformation gradients across the 
planes defined by $n_1$ and $n_2$. Note that $W$ is zero when the deformation gradient lies at $SO(2) V$ or at $SO(2) I$. Now we define $y_0$ in eq.\eqref{sim:energMin} such that
\begin{align}
    \nabla y_0 = 0.5 I + 0.5 R_1 V,
\end{align}
and we initialise $y_h$ by the interpolation $y_h = I_h^1 y_0$, where 
$I_h^1 : W^{1,p}(\Omega) \rightarrow V_h^1(\Omega)$.
Discretising the domain through equal squares with normals $n_1$ and $n_2$, where each rectangle is divided to four triangles, perfect interfaces between 
discontinuous deformation gradients (phases) can occur due to this triangulation. For a comparison of meshes with or without the above property see 
\cite[Section 4.5]{grekasPhD}. 
Minimising the discrete version of eq.~\eqref{sim:energMin} the total potential energy 
and $||y_h - y_0||_{L^2(\Omega)}$ decrease 
for smaller mesh sizes, fig.~\ref{energy_over_resolution}.  The energy is reduced through a minimising 
sequence as it is depicted in fig.~\ref{fig:minimizing_seq_def}, describing the deformed state. 
  In some regions perfect interfaces with 
zero contribution are formed between $\nabla y_h=I$ and $\nabla y_h = R_1V$ 
across the direction $n_1 = (1,0)$  and between $\nabla y_h = I$ and $\nabla y_h = R_2 V$ 
across the direction $n_2 = (0,1)$. These regions illustrate the minimising sequence in the reference configuration, 
 see fig.~\ref{fig:minimizing_seq}.
Note that $R_1 V$ is not rank-one connected to $R_2 V$, $R_1 \ne R_2$. Hence, a transition layer
is formed between these two deformations, green color in fig.~\ref{fig:minimizing_seq}. This is the only 
non-zero part of the energy. As finer and finer microstructures appear on each side of a green curve 
$y_h \rightarrow y_0$ in $L^\infty(\Omega)$ as $h\rightarrow 0$, which means a ``macroscopic'' deformation is formed equal to $y_0$ and 
the energy at the transition layer tends to zero as $h \rightarrow 0$, see also fig.~\ref{fig:overLine}. In this case, the discrete 
minimiser captures a minimising sequence $\{ y_k\} \in W^{1,\infty}(\Omega)$ such that $y_k\stackrel{\ast}{\rightharpoonup} y_0$ in $W^{1, \infty}(\Omega)$
where $\{ \nabla y_k\} $  generates the homogeneous Young measure 
$\nu_x = 0.5 \delta_I + 0.5 \delta_V$. Note that $y_0$ is the minimiser of the relaxed problem.

The above simulations have been performed by choosing $\alpha=80$ in eq.~\eqref{eq:energy_discrete}. Note that, for the proposed multi-well energy 
$$ -1 +  |F|^8  \lesssim W(F) \lesssim  1 +  |F|^8.$$
We have noticed that when $p=8$ roundoff errors affects the discrete minimisation process. Note that
the second term of the proposed penalty, eq.\eqref{eq:penalty}, is getting less important numerically, due to the fact that the 
power of $1/8$ appears and small numbers are described by finite precision. One can rewrite the jump terms of eq.~\eqref{eq:energy_discrete} as
\begin{equation}
\begin{aligned}
jumps =   \sum_{e\in E_h} h_e \int_e \left|\frac{\jumpop{u_h}}{h_e} \right|^p ds, 
\end{aligned}
\end{equation}
and the penalty term as 
\begin{equation}
\begin{aligned}
\alpha{\rm Pen}(u_h) = \frac{\alpha}{2}\left(\int_\Omega 1 + W(\nabla u_h)dx + jumps \right)^\frac{p-1}{p} \left( 2^p jumps \right)^\frac{1}{p}.
\end{aligned}
\end{equation}
Under the above modifications we have obtained minimisers that capture minimising sequences of the continuous problem. 
The above minimising sequence appears for both $y_h \in V_h^q(\Omega) \cap C^0(\Omega)$ and $y_h \in V_h^q(\Omega)$, where both
converge to the minimiser of the relaxed problem, since the proved convergence results hold for conforming finite element spaces as well.

\begin{figure}
\centering
 \includegraphics[width=0.7\textwidth]{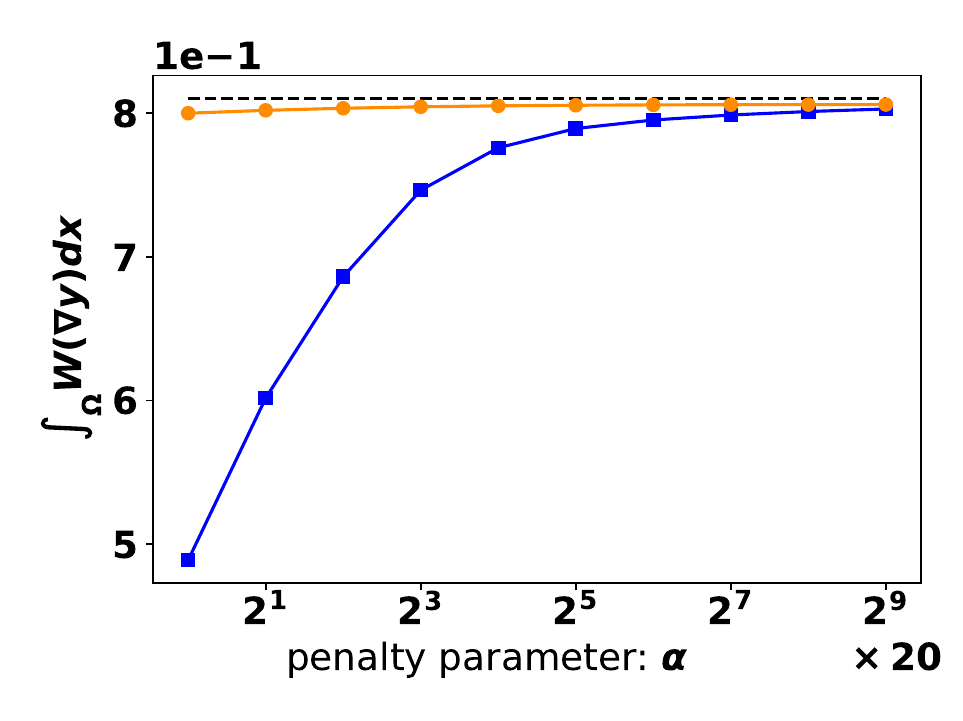}
  \caption{Total potential energy (Vertical axis) for the homogeneous deformation $y_0(x) = F_0 x$ (dashed black line),
  for the computed minimiser with respect to the stabilisation parameter $\alpha$ employing the penalty of eq.~\ref{eq:convexpen}
  (blue squares) and the penalty of eq.~\ref{eq:penalty} (orange circles).
  } 
  \label{W_over_alpha}
\end{figure}

\begin{figure}
\centering
\subfloat[]{\hspace{-0.8cm} \includegraphics[width=0.55\linewidth]{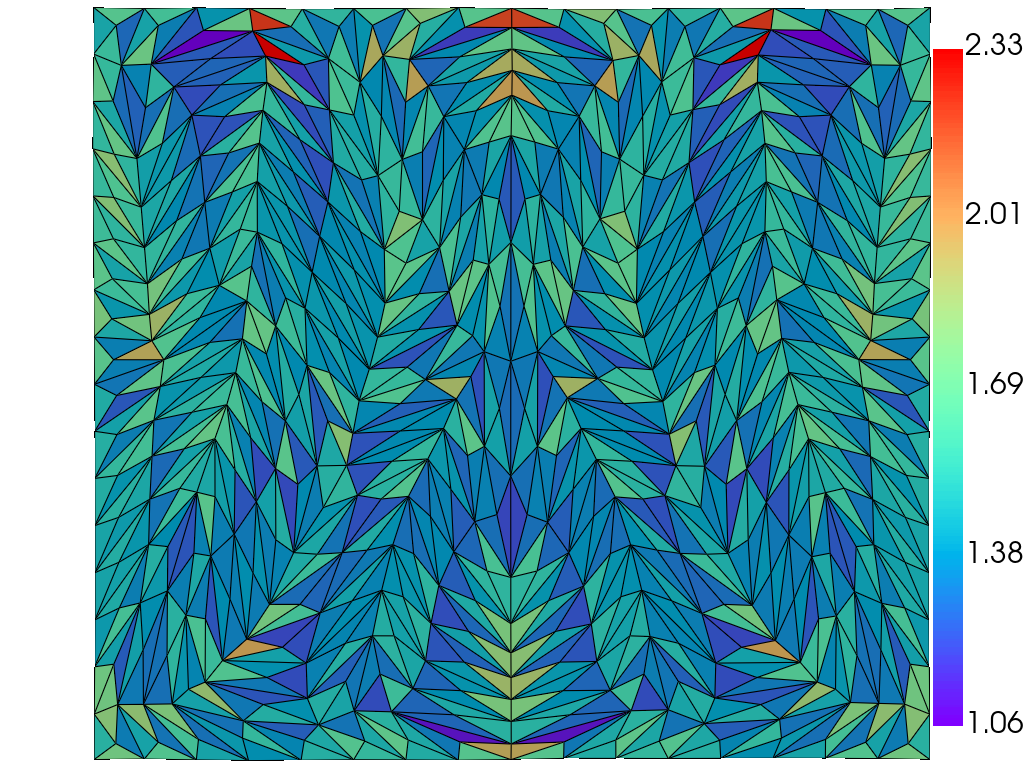} 
\label{fig:convex_a1} }
\subfloat[]{\hspace{-0.2cm} \includegraphics[width=0.55\linewidth]{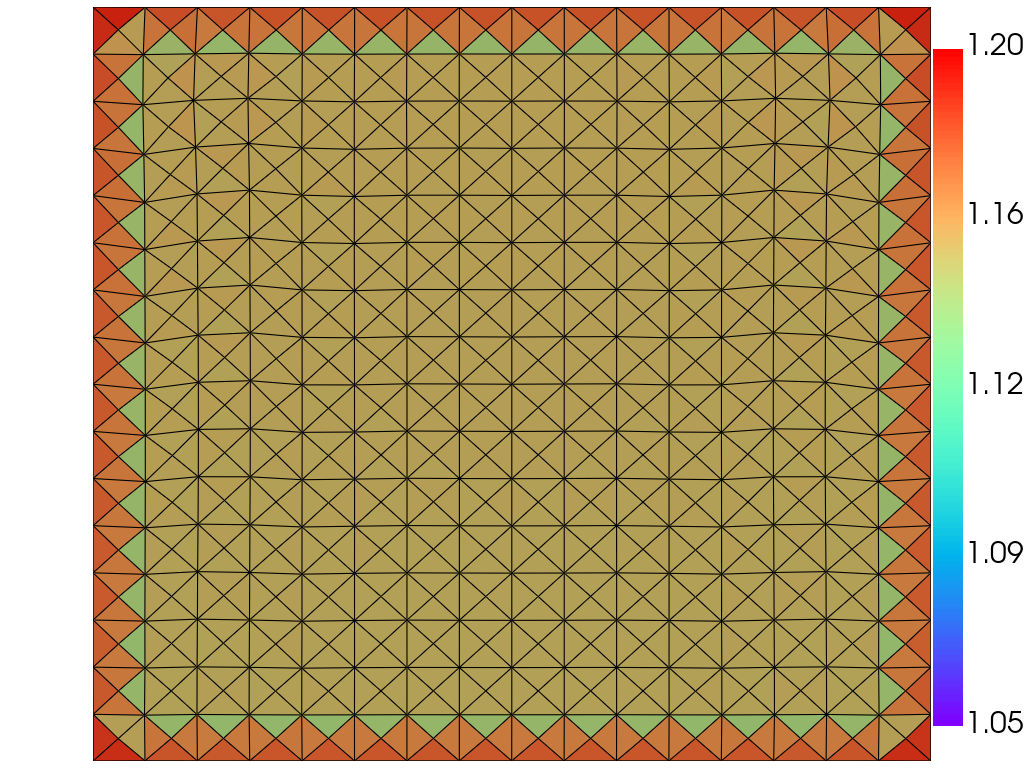}
\label{fig:convex_a4} }
\\
\subfloat[]{\hspace{-0.7cm}\includegraphics[width=0.55\linewidth]{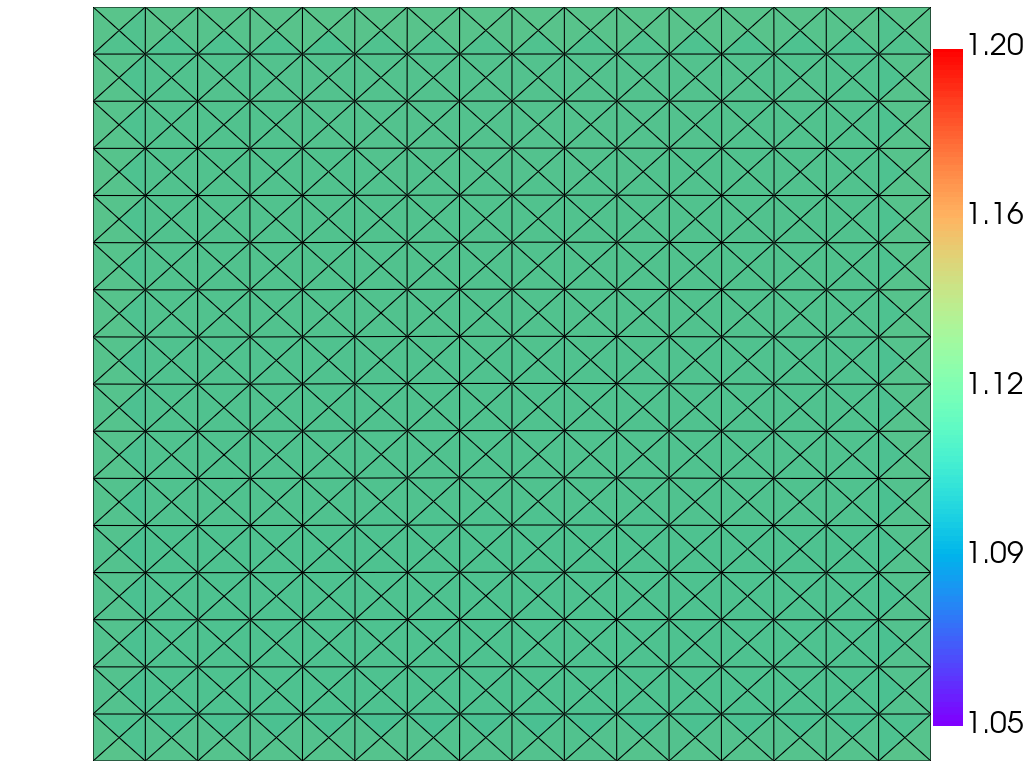} 
\label{fig:quasiconvex_a1} }
\subfloat[]{\hspace{-0.2cm}\includegraphics[width=0.55\linewidth]{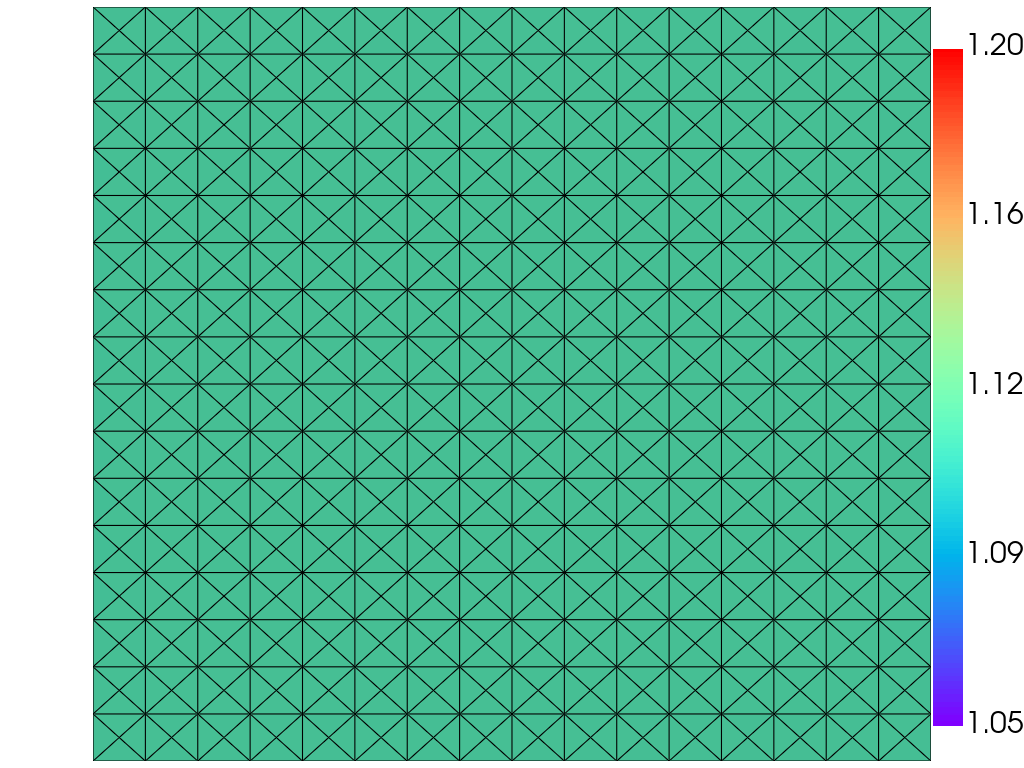} 
\label{fig:exact} }
\caption{ Density ($1/\det \nabla{y}$) in the deformed configuration under uniaxial compression (10\% strain). 
Employing the penalty of eq.~\eqref{eq:convexpen}, the density for the computed minimisers are illustrated: \protect\subref{fig:convex_a1} when $\alpha=20$ and 
\protect\subref{fig:convex_a4} when $\alpha =160$ (see blue squares of Figs.~\ref{fig:fn1c} and \ref{fig:fn1d} at $\alpha =20, 160$).
\protect\subref{fig:quasiconvex_a1}: Computed solution when the proposed penalty of eq.\eqref{eq:penalty} is used for $\alpha=20$. \protect\subref{fig:exact} Density of the exact homogeneous minimiser $y_0$. 
}
\centering
\label{fig:compression_dens}
\end{figure}

\begin{figure}
\centering
  \subfloat[]{ \includegraphics[width=0.5\linewidth]{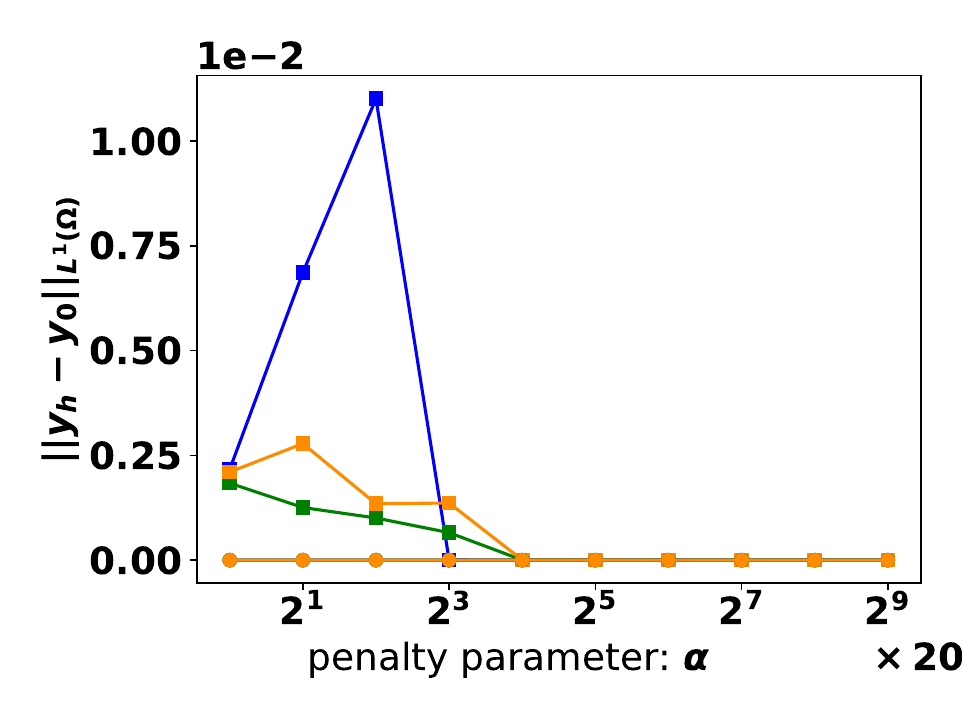}
  \label{fig:fn1c} }
 \subfloat[]{ \includegraphics[width=0.5\linewidth]{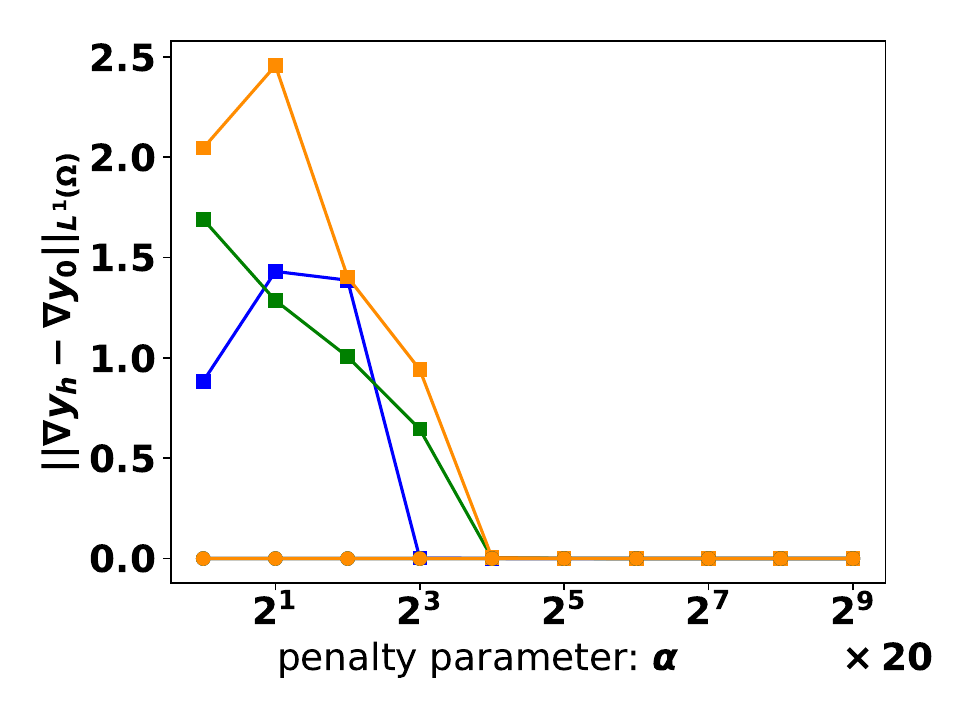}
 \label{fig:fn1d} 
  }
  \caption{
  Comparing the accuracy  of the two proposed stabilisation penalty terms \eqref{eq:penalty} (circles) and 
  \eqref{eq:convexpen} (squares) for   
  the polyconvex $W(F) = |\det F|^2$  strain energy under
  10\% uniaxial compression.
  Horizontal axis: values of the penalty parameter $\alpha$. 
  Vertical axes: $|y_h - y_0|_{L^1(\Omega)}$ and  $|y_h - y_0|_{W^{1,1}(\Omega)}$ errors 
  for various mesh resolutions, specifically for 1024 (blue), 2034 (green) and 4096 (orange) triangles,
  where $y_h$ denote the numerical solutions and $y_0$ the  homogeneous minimiser. 
  In \protect\subref{fig:fn1c} circular error belongs in the range of 
  $10^{-8}$ and $10^{-9}$, while
  in   \protect\subref{fig:fn1d} circular errors are of the order 
  $10^{-6}$.
 } 
  \label{fig_errors}
\end{figure}

\begin{figure}
\centering
 \includegraphics[width=0.7\textwidth]{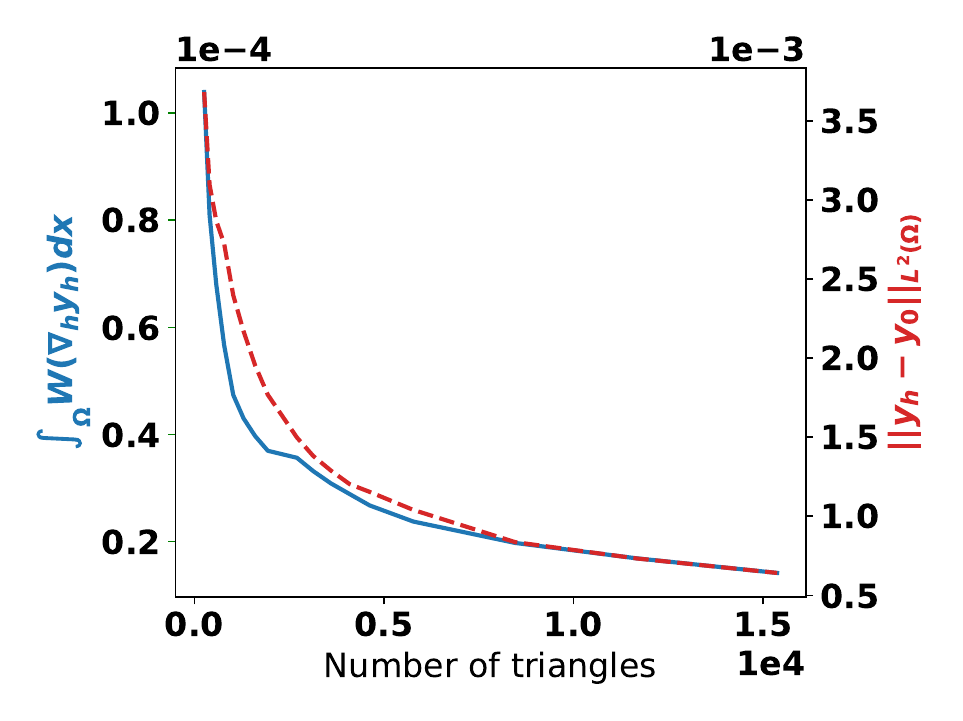}
  \caption{Imposing the homogeneous boundary conditions $y_0(x) = (0.5R_1 V
  +0.5 I)x$, minimisers are computed varying the mesh resolution (horizontal axis).
  Blue curve denotes the
  total potential energy (scale of left vertical axis) and red dashed curve is the 
  error $||y_h - y_0||_{L^2(\Omega)}$ (scale of right vertical axis)
  where $y_h$ denote the computed minimiser.
  } 
  \label{energy_over_resolution}
\end{figure}

\begin{figure}
\centering
 \subfloat[]{\hspace{-0.8cm} \includegraphics[width=0.55\linewidth]{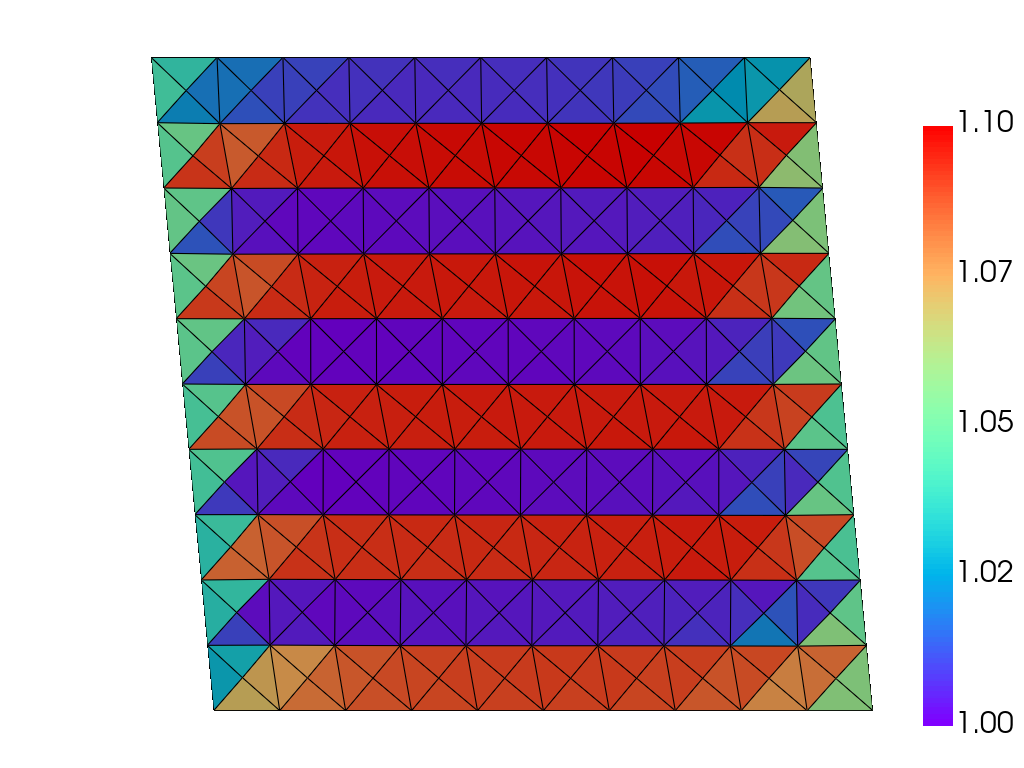} 
\label{fig:min_seq16} }
\subfloat[]{\hspace{-0.2cm} \includegraphics[width=0.55\linewidth]{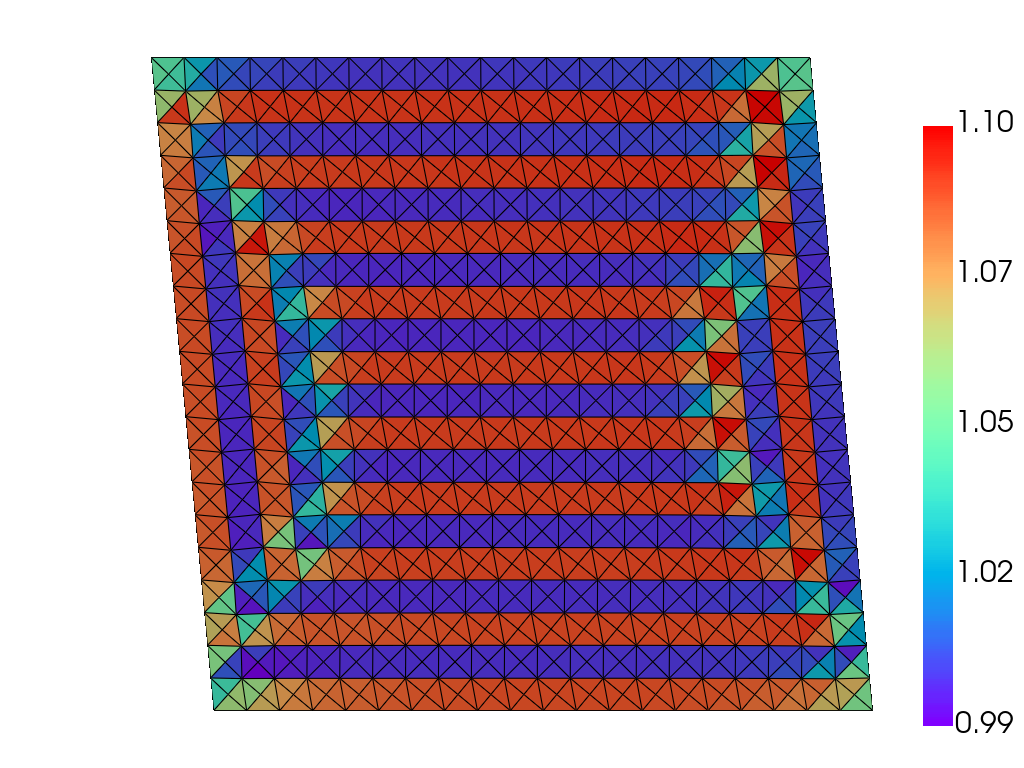}
\label{fig:min_seq32} }
 \\
\subfloat[]{\hspace{-0.7cm}\includegraphics[width=0.55\linewidth]{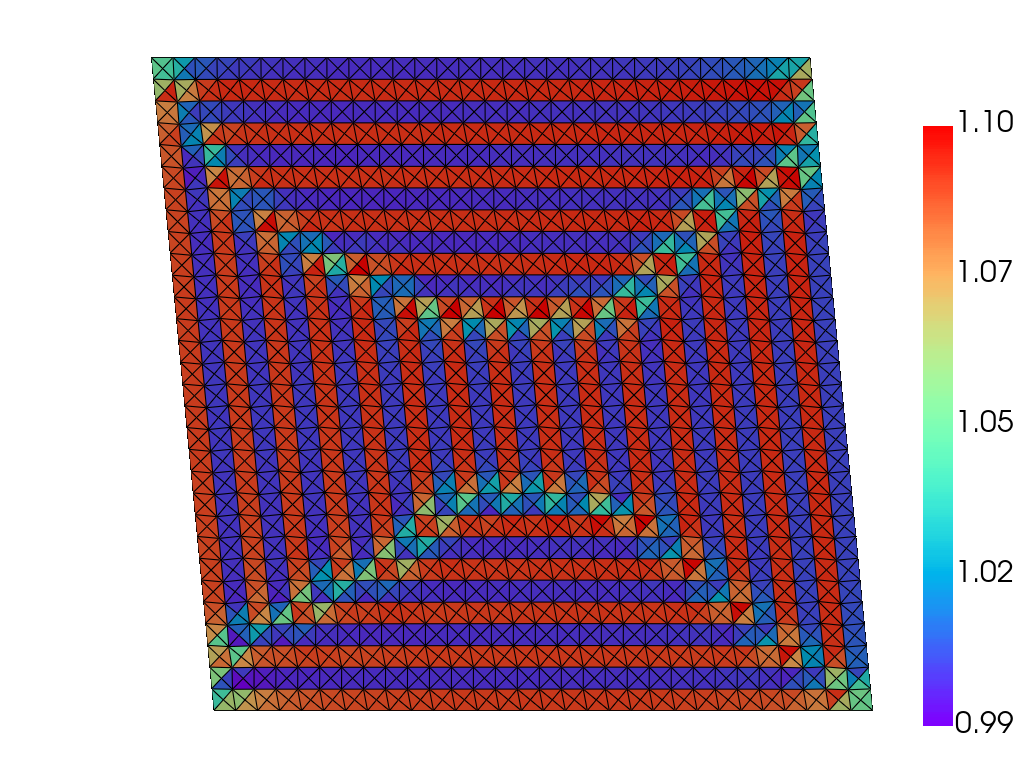} 
\label{fig:min_seq64} }
 \subfloat[]{\hspace{-0.2cm}\includegraphics[width=0.55\linewidth]{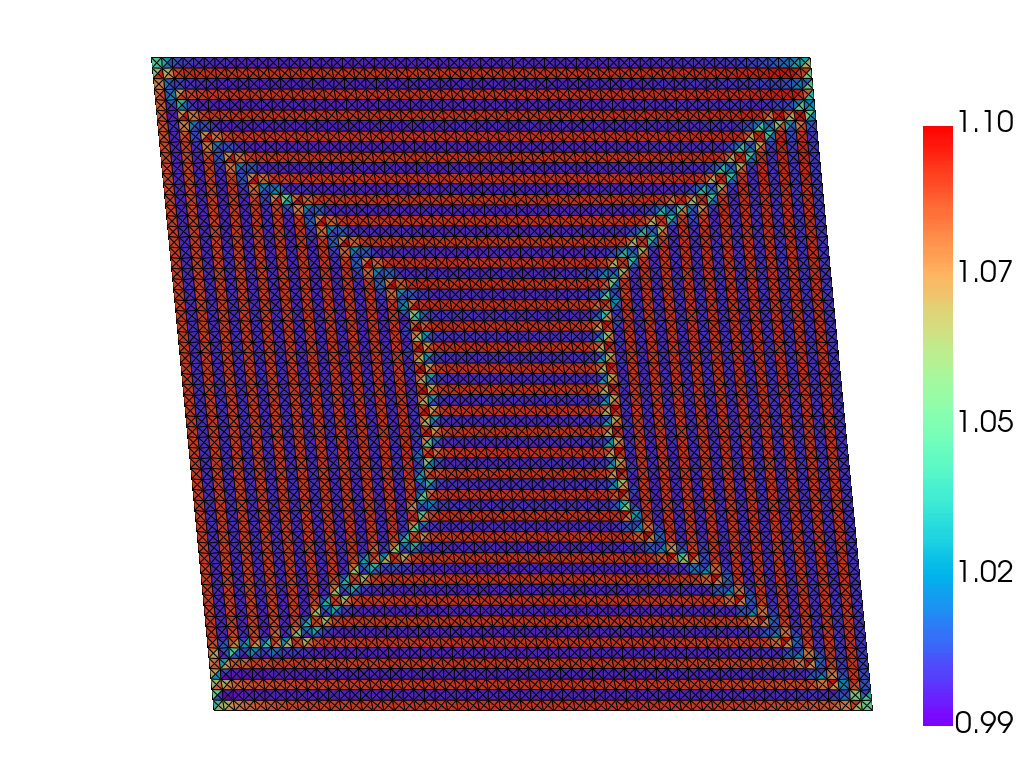} 
\label{fig:min_seq128} }
 \caption{The maximum eigenvalue ($\lambda_{max}$) in the deformed configuration of the computed minimiser for each triangle. The wells $I, V$ have the
 eigenvalues $(1,1)$ and $0.9, 1.09871212$ respectively. 
 In blue regions $\lambda_{max} =1$ which indicates that well $I$ is attained up to a rotation. 
 Similarly, red color, $\lambda_{max} =1.09871212$, corresponds to $R_i V$ up to a rotation, $i=1,2$.
  }
 \centering
 \label{fig:minimizing_seq_def}
\end{figure}

\begin{figure}
\centering
 \subfloat[]{\hspace{-0.8cm} \includegraphics[width=0.55\linewidth]{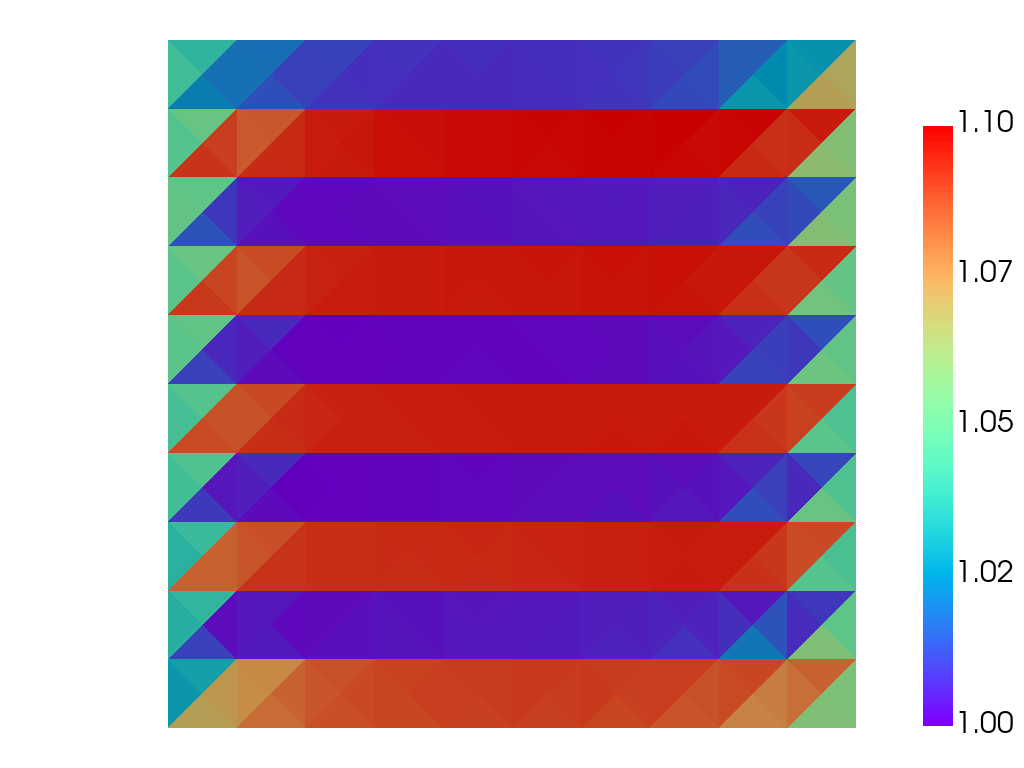} 
\label{fig:min_seq16} }
\subfloat[]{\hspace{-0.2cm} \includegraphics[width=0.55\linewidth]{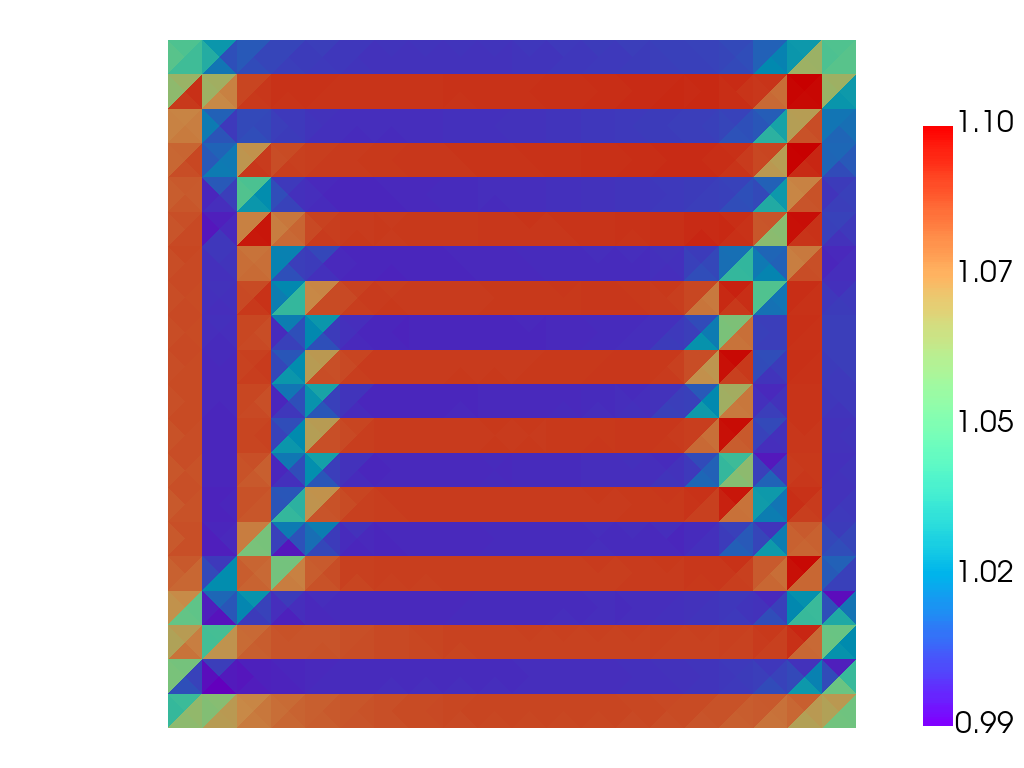}
\label{fig:min_seq32} }
 \\
\subfloat[]{\hspace{-0.7cm}\includegraphics[width=0.55\linewidth]{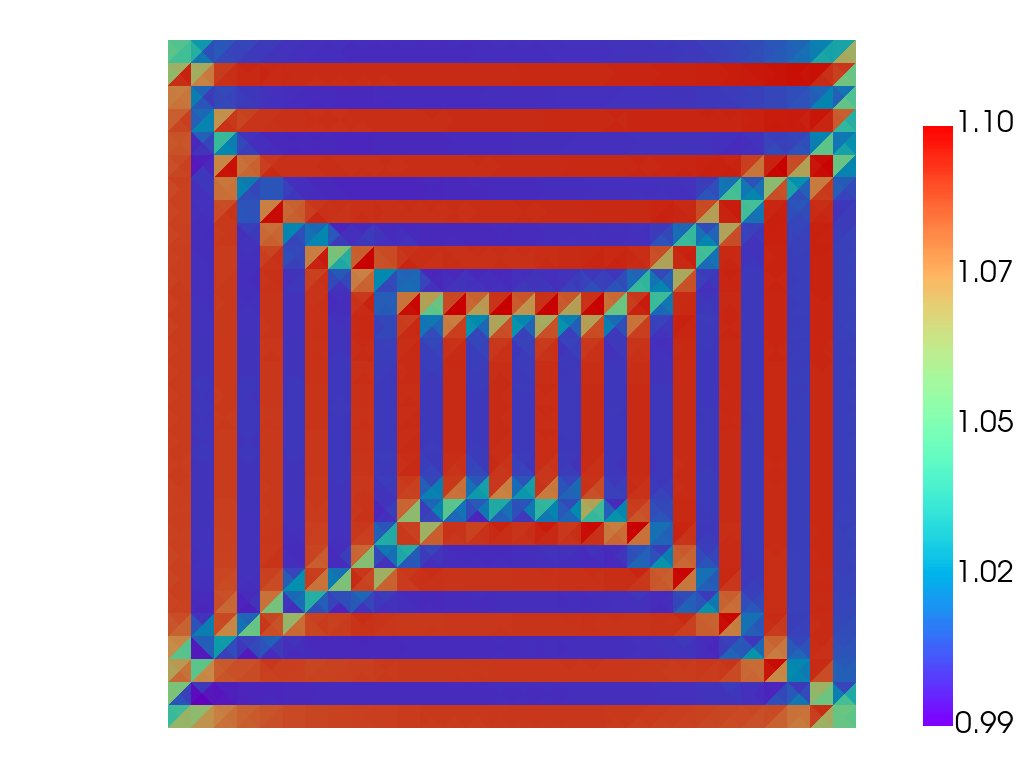} 
\label{fig:min_seq64} }
 \subfloat[]{\hspace{-0.2cm}\includegraphics[width=0.55\linewidth]{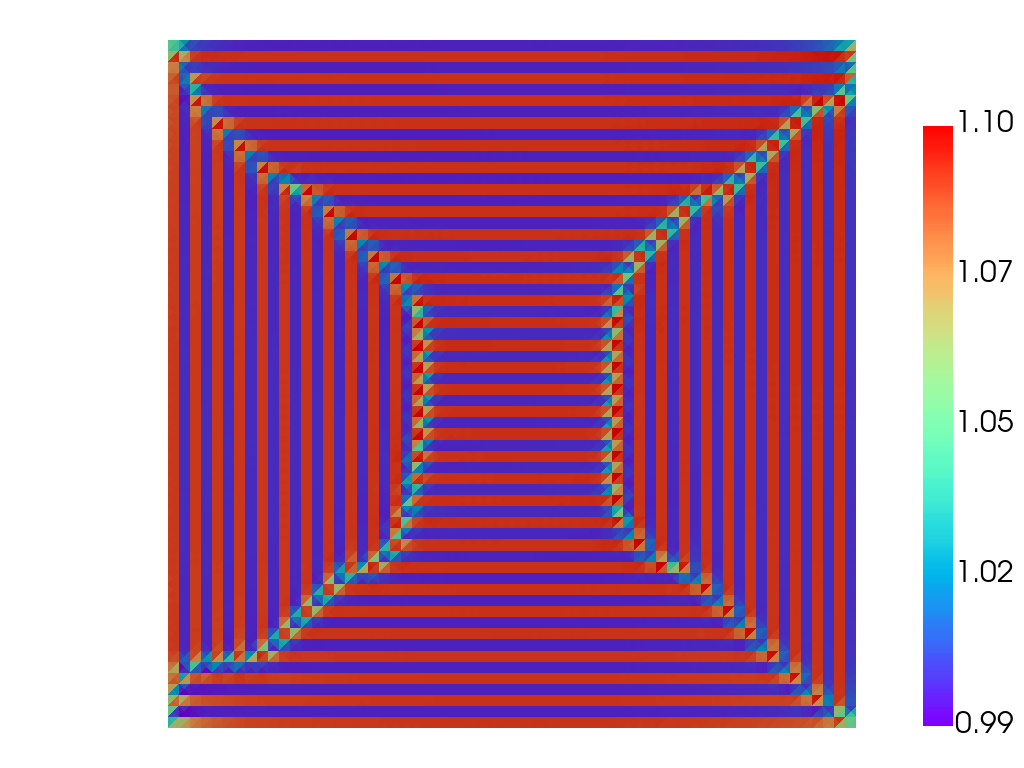} 
\label{fig:min_seq128} }
 \caption{The maximum eigenvalue ($\lambda_{max}$) of the discrete minimiser for each triangle expressed 
 in the reference configuration. The wells $I, V$ have the
 eigenvalues $(1,1)$ and $0.9, 1.09871212$ respectively. 
 In blue regions $\lambda_{max} =1$ which indicates the well $I$ appears up to a rotation. 
 Similarly, in red color $\lambda_{max} =1.09871212$ correspond  to $R_i V$ up to a rotation, where i=1 when the interface between the two phases is 
 normal to $n_1 =(1,0)$ and i=2 when the interface is normal to $n_2 =(0,1)$. 
  }
 \centering
 \label{fig:minimizing_seq}
\end{figure}

\begin{figure}
\centering
 \subfloat[]{\hspace{-0.8cm} \includegraphics[width=0.55\linewidth]{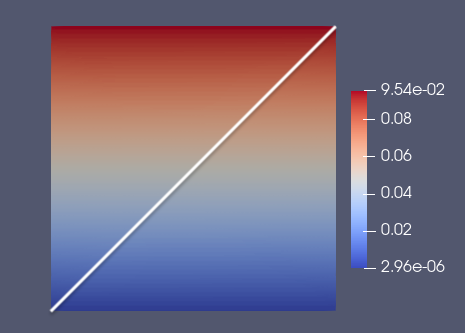} 
\label{fig:overLineu} }
\subfloat[]{\hspace{-0.2cm} \includegraphics[width=0.55\linewidth]{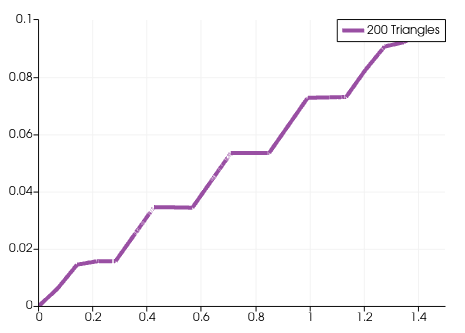}
\label{fig:overLine10} }
 \\
\subfloat[]{\hspace{-0.7cm}\includegraphics[width=0.55\linewidth]{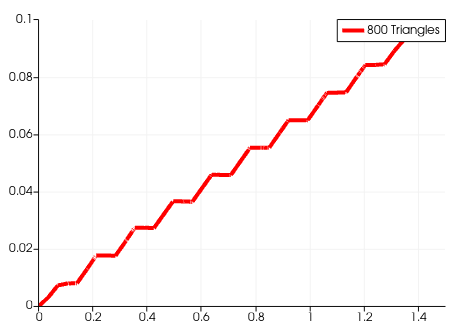} 
\label{fig:overLine20} }
 \subfloat[]{\hspace{-0.2cm}\includegraphics[width=0.55\linewidth]{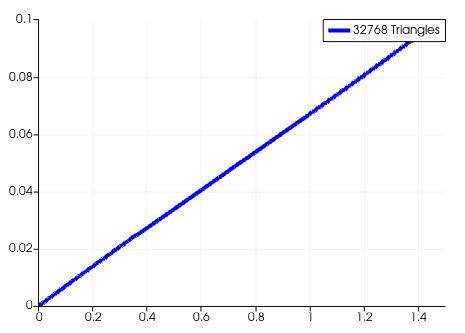} 
\label{fig:overLine128} }
 \caption{  Displacement of the computed minising sequence varying the mesh resolution. 
 \protect\subref{fig:overLineu}: Pointwise values of $|u| = \sqrt{u \cdot u}$, $u : \Omega \rightarrow \R^2$ is the displacement vector. 
 For 200, 800, and 32768 triangles the values of $|u|$ along the diagonal of the domain (white line in  \protect\subref{fig:overLineu})  are illustrated in 
 \protect\subref{fig:overLine10}, \protect\subref{fig:overLine20} and \protect\subref{fig:overLine128} respectively.}
 \centering
 \label{fig:overLine}
\end{figure}

 \clearpage

\printbibliography

\end{document}